\documentclass[12pt]{amsart}

\usepackage{color,graphicx,amssymb,latexsym,amsfonts,txfonts,amsmath,amsthm}
\usepackage{mathrsfs}
\usepackage{wrapfig}
\usepackage{epsfig}
\usepackage{psfrag}
\usepackage{array} % añadido para tablas
\usepackage{multirow}% añadido para tablas
\usepackage{tikz} % añadido para tablas
\usepackage{setspace} % para el espacio de las filas de la tabla
\usepackage{rotating} % para rotar una tabla
\usepackage{amssymb}
\usepackage{graphicx}
\usepackage{pb-diagram}
\newtheorem{thm}{Theorem}[section]

\newtheorem{lem}{Lemma}[section]
\newtheorem{prop}{Proposition}[section]

\def\D{{\mathcal D}}

\def\Z{{\mathbb Z}}
\def\O{{\mathcal O}}
\def\B{{\mathcal B}}
\def\G{{\mathcal G}}
\def\M{{\mathcal M}}
\def\T{{\mathcal T}}

\def\g{\gamma}
\def\S{\Sigma}
\def\s{\sigma}

\def\b{\beta}
\def\a{\alpha}
\def\d{\delta}

\def\Im{{\rm Im}\,}

\def\DArriba{{\mathcal D}_{\tilde\Sigma}}
\def\DAbajo{{\mathcal D}_{\Sigma}}

\date{\today}

\begin{document}

\title[Boundary of equisymmetric loci of Riemann surfaces with abelian symmetry]
 {Boundary of equisymmetric loci of Riemann surfaces with abelian symmetry}
\author{Raquel D\'{\i}az}
\address{Departamento de Álgebra, Geometr\'{\i}a y Topolog\'{\i}a. Facultad de Ciencias Matem\'aticas. Universidad Complutense de Madrid. Espa\~na}
\email{radiaz@mat.ucm.es}

\author{V\'ictor Gonz\'alez-Aguilera}
\address{Departamento de Matem\'atica. Universidad T\'ecnica Federico Santa Mar\'{\i}a. Valpara\'{\i}so, Chile}
\email{victor.gonzalez@usm.cl}
\thanks{\noindent
 {\it 2020 Mathematics Subject
    Classification}: Primary 32G15; Secondary 14H10.\\
%  \mbox{\hspace{11pt}}
  {\it Key words}: Moduli space, stable curves, equisymmetric stratification.\\
   %\mbox{\hspace{11pt}}
The first author was partially supported by  Project PID2020-114750GB-C32/AEI/10.13039/501100011033. }

\begin{abstract}
Let ${\mathcal M}_g$ be the moduli space  of compact connected Riemann surfaces of genus $g\geq 2$
 and  let $\widehat{{\mathcal M}_g}$ be  its  Deligne-Mumford compactification, which is stratified by the topological type of the stable Riemann surfaces.  We consider the equisymmetric loci in $\M_g$ corresponding to Riemann surfaces whose automorphism group is abelian and determine the topological type of the maximal dimension strata at their boundary.  For the  particular cases of the hyperelliptic and the cyclic $p$-gonal actions, we describe all the topological strata at the boundary in terms of trees with a fixed number of edges. 
 
\end{abstract}
\maketitle

\maketitle

\section{Introduction}

Let ${\mathcal M}_g$ be the moduli space of compact connected Riemann surfaces of genus $g\geq 2$ up to biholomorphism or, equivalently,   the space of hyperbolic surfaces of genus $g$ up to isometry, or the space of smooth projective curves of genus $g\geq 2$ defined over ${\mathbb C}$. As a space, it is easier to consider the Teichm\"uller space $\T_g$, i.e., the space of {\it marked} hyperbolic  surfaces of genus $g$, since it  is homoeomorphic to an open ball of $\mathbb C^{3g-3}$. The mapping class group ${\rm Mod}_g$ acts on $\T_g$ by changing the marking of a surface, and therefore the quotient $\T_g/ {\rm Mod}_g$  can be  identified to  $\M_g$. 
 This action is properly discontinuously and thus $\M_g$ is an orbifold whose  singular points   (when $g>2$) are those Riemann surfaces which  admit non-trivial automorphisms. 
This subset  is called the {\it  branch locus} and is denoted $\B_g$. It can be  stratified in a natural way according to the topological action of the isometry group of the points in $\B_g$, the strata are called {\it equisymmetric strata} or {\it loci} (see Section \ref{sec:TopActions}).  This stratification is the object of intense study, both determining all the strata and studying topological properties as the connectivity of $\B_g$. For instance Broughton \cite{Broughton-Classifying} provides    the 
classification of topological actions in genus 2 and 3    and Kimura \cite{kimura} that in genus  4. In a series of papers, Costa, Izquierdo and  Bartolini (see for instance \cite{costa0}, \cite{bartolini}, \cite{bartolini2}) study the orbifold structure of moduli space and the connectivity of its branched locus.

 On the other hand, Teichm\"uller space can be augmented by adding marked  stable Riemann surfaces (\cite{abikoff}, \cite{bers}, Section \ref{sec:Prelim-Augmented}). The resulting space $\widehat{\T_g}$ is not compact, but its quotient  $\widehat{\M_g}$ by the mapping class group is compact  and  coincides with the Deligne-Mumford compactification of moduli space, that is,  the set of stable curves, i.e. curves which have only nodes as singularities and with finite group of automorphisms   (\cite{deligne}, see \cite{Hubbard} for a detailed exposition on this). Intuitively, this compactification is saying  that the only  way of going to the infinity at moduli space is that the length of a multicurve (a subset of disjoint simple closed curves) is going to zero. 
 
 The previous  construction provides  a stratification to the  boundary of moduli space $\partial \M_g=\widehat{\M_g}\setminus \M_g$, each strata consisting on  a subset of stable Riemann surfaces with the same topological type.  
We will refer to this stratification as the topological stratification of $\partial \M_g$ and to its strata as {\it topological strata}.  It can be noticed that the topological type of a stable Riemann surface  is captured by its associated weighted stable graph.

A natural question  to ask is  which topological strata of $\partial \M_g$ appear at the boundary of a given equisymmetric locus $\mathcal E$. If $X_n$ is a sequence in  $\mathcal E$ going to the boundary, there is a curve   whose length along the sequence is going to 0. Since $X_n$ has   non-trivial  isometries, there is actually an equivariant multicurve whose length is going to zero. Quotienting by the isometry group, we obtain a curve going to zero along the sequence of  quotient orbifolds $\O_n$. Now, this process can be reversed: starting  with a curve  in the quotient orbifold and making its length going to zero produces an equivariant multicurve in $X_n$ with length going to zero. Decreasing to 0 the length of a multicurve is always  possible provided that all the components of the complement of the multicurve are hyperbolic. These multicurves are called {\it admissible}. 
Summing up, to find the topological strata in the boundary of an equisymmetric stratum $\mathcal E$ reduces to find the preimage of admissible multicurves  under a branched covering. 
In \cite{diaz1}, we addressed this question and  described a procedure to solve it. Also in that paper and in \cite{diaz2}, we applied this method to find all the topological strata at the boundary of the pyramidal equisymmetric locus. 
The case of the 1-complex dimensional equisymmetric strata for a family uniformized by a Fuchsian group of signature $s = (2, 2, 2, m), m \geq 3$ is studied in \cite{costa2}. In \cite{costa1} it is  shown that the completion of the 3-gonal equisymmetric space in $\widehat{\M_g}$ is connected for $g\geq 5$. 
In recent papers (\cite{BCI-1}, \cite{BCI-2}) Broughton, Costa and Izquierdo have studied the topology of the one-dimensional equisymmetric loci, which are punctured Riemann surfaces. Some of these punctures correspond to limiting stable, nodal surfaces in  $\widehat{\M_g}$, i.e., the topological strata at the boundary of the equisymmetric locus. 
An algebraic stack treatment of boundary of loci of curves can be seen in \cite{johannes}.

In this paper we focus on  the equisymmetric  strata corresponding to abelian actions. 
 For the  hyperelliptic and  the $p$-gonal actions, we  provide easy descriptions of the topological strata at the boundary of the corresponding equisymmetric loci in terms of trees with a bounded number of edges.    For any abelian action, we determine the topological strata of maximal dimension at the boundary of  the corresponding  equisymmetric locus.

\bigskip 

The structure of this paper is the following:
  In Section \ref{sec:MainThmInDiGo} we recall the main result in \cite{diaz1}: given a topological action $(G,\O,\Phi)$ on a surface $S$ 
  and an  admissible  multicurve $\S$ in $\O=S/G$, the dual graph of  its preimage  $p^{-1}(\S)$ is determined. 
  A first important simplification of this result occurs when the group $G$ is abelian, since we can work with homology instead homotopy. This will be seen  in Section \ref{sec:Simplification-Abelian}.
  A second simplification occurs when the   orbifold $\O$ is topologically a sphere, this case is studied in Section \ref{sec:genusO=0}.
  Both simplifications are present in the $p$-gonal and hyperelliptic equisymmetric strata. We will study these cases in sections \ref{sec:p-gonal} and \ref{sec:Hyperelliptic}. In particular, we find all the topological strata at the boundary of the equisymmetric loci corresponding to all  the cyclic $p$-gonal actions up to genus 4.    
  Finally, in section \ref{sec:Abelian}, we will study the case of $G$ any abelian group and the multicurve  $\S$ consisting on a single curve. This corresponds to finding the maximal dimension topological strata in the boundary of the equisymmetric locus.

\section{Preliminaries}\label{sec:Prelim}

\subsection{Equisymmetric loci of moduli space}\label{sec:TopActions}

We recall definitions and fix notation. For a general reference in this subject  see \cite{Broughton-Equisymmetric}. 
Consider an action  of a finite group $G$ on a closed  surface $S$, that is, a monomorphism    $\iota\colon G\to {\rm Hom}^+(S)$ into the group of preserving orientation homeomorphisms of $S$. It has an associated   
regular branched covering $p\colon S\to \O$ over the set of orbits $\O$,  and an associated epimorphism
 $\Phi\colon \pi_1(\O,*)\to G$ that assigns to a loop $\alpha$ in $\O$ the deck transformation that sends one end of a lift of $\alpha$ to the other one. 
The space $\O$ is a closed  orbifold whose singular locus consists on a finite  number of cone points. 
 
 The Euler characteristics of $S$ and $\O$ are related by the Riemann-Hurwitz formula $\chi(S)=|G|\chi(\O)$. We recall that Euler characteristic of an orbifold of signature $(g,c; m_1,\dots, m_r)$ is  $\chi(\O)=2-2g-c-\sum_{i=1}^r(1-\frac{1}{m_i})$, where $g$ is the genus of $\O$, $c$ the number of its boundary components and $m_1, \dots, m_r$ are the orders of its cone points. We remark that in the above situation, that is, if $\O$ is the quotient of a closed surface by the action of a group, the orbifold $\O$ does not have boundary components. Nevertheless,  along the paper we will need to use orbifolds with boundary and to compute their Euler characteristic.

  Any epimorphism $\Phi\colon \pi_1(\O,*)\to G$ whose kernel is the fundamental group of a surface  $S$ ({\it surface kernel epimorphism}) determines an action of $G$ on $S$. We denote this action by $(G,\O, \Phi)$.

Two actions 
$\iota\colon G\to {\rm Hom}^+(S)$ and , $\iota'\colon G\to {\rm Hom}^+(S')$, are  {\it topologically equivalent} if   there is an automorphism  $  \omega\in Aut(G)$ and a homeomorphism  $ h\colon S\to  S'$ such that 
 $\iota'(\omega(g))=h\circ \iota(g)\circ h^{-1}  $ for all $g\in G$.

The {\it equisymmetric stratum} or {\it locus} determined by the topological class of action represented by  $(G,\O,\Phi)$ is the subset $ \M_g(G,\O,\Phi)$ of points $X \in \M_g$ such that the action of $Aut^+(X) $ is topologically equivalent to the action   $(G, \O, \Phi)$.

  Broughton  (\cite{Broughton-Classifying})  determined all the topological actions in surfaces $S$ of  genera $g=2,3$  and Kimura (\cite{kimura}) determined those of genus $g=4$. We will use some of these topological actions  to exemplify our results.

\subsection{Admissible multicurves and weighted graphs.}\label{sec:WeightedGraphs}

Let   $\O$  be an orientable    2-orbifold, i.e., one  of signature $(g,c; m_1,\dots,m_r )$, where $g$ is the genus, $c$ is the number of connected components and $m_1,\dots, m_r$ are the orders of the cone points.   

\noindent
{\bf Definition.} Let $\O$ be an orbifold with $\chi(\O)<0$. An  {\it (admissible) multicurve} in $\O$ is a collection  $\Sigma =\{\g_1,\dots, \g_k\}$ of   simple closed curves or simple arcs joining pairs of cone points of order 2, all of them disjoint, and  such that   each component of $\O\setminus
\S$  has negative Euler characteristic.

We will represent multicurves in orbifolds by weighted graphs, defined next.

\noindent
{\bf Definition.} {\it (Weighted graph dual to a multicurve)}
Let $\S$ be a multicurve in a closed  orbifold $\O$. Its {\it dual graph}  $\D_{\S}$ is defined as follows:
\begin{enumerate}
\item its vertices are the connected components of $\O\setminus \S$;
\item its edges are the closed curves  in $\S$;
\item its semiedges are the arcs in $\S$;
\item an edge $ \g\in \S$ connects two vertices $\O_{j_1},\O_{j_2}$ if and only if $\g$ is in the boundary of $\O_{j_1}$ and $\O_{j_2}$; a semiedge $\g'\in \S$ is connected to the vertex $ \O_j$ if the arc $\g'$ is in the boundary of $\O_j$.
\item if $V=\O_j$ is a vertex, its weight is $w(V)=(g(\O_j);m_{i_1},\dots, m_{i_j})$ where $g(\O_j)$ is the genus of $\O_j$ and $m_{i_k}$ are the orders of the cone points of $\O_j$.

\end{enumerate}

Figure \ref{fig:DualGraph} shows a multicurve in an orbifold and its dual graph.

It is clear from the definition that the graph $\D_{\S}$ determines the topological information of the  pair $(\O,\S)$. 

\psfrag{2}{\small 2}
\psfrag{5}{\small 5}
\psfrag{7}{\small 7}
\psfrag{(0;)}{\small  $(0;\,)$}
\psfrag{(1;7)}{\small  $(1;7)$}
\psfrag{(0;)}{\small  $(0;\,)$}
 \psfrag{(0;5)}{\small  $(0;5)$}
 \psfrag{(0;3)}{\small  $(0;2)$}
\begin{figure}
\includegraphics[height=3.5cm,width=10cm]{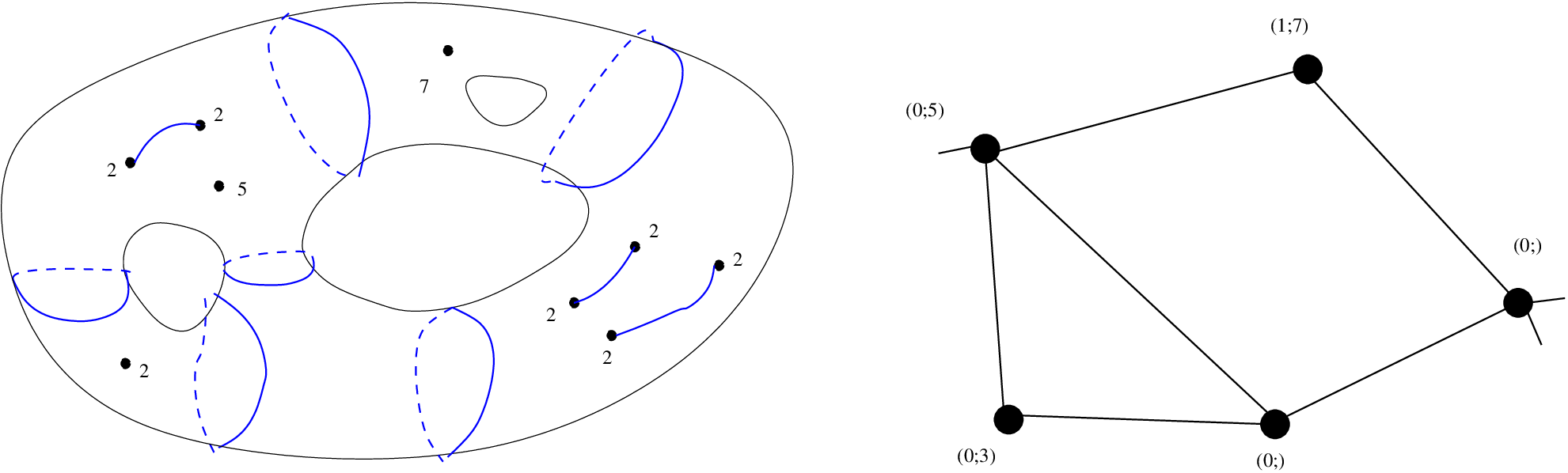}
\caption{Admissible multicurve and its dual graph}
\label{fig:DualGraph}
\end{figure}

 When the orbifold is a surface (there are no cone points), a multicurve is just a finite union of simple closed curves, and its dual graph is a weighted graph with no semiedges and  where the weight of each vertex  is the genus of the corresponding subsurface.

\subsection{Augmented moduli space $\widehat{\M_g}$}\label{sec:Prelim-Augmented}
Let $\S$ be an admissible multicurve on a surface $S$. The graph dual to $\S$ is called {\it stable graph}. The space $S/\S$ where each component of $\S$ has been collapsed to a point is called a {\it stable surface}. Each collapsed curve is called a {\it node}, and the complement of the nodes are called  {\it parts}. 
A {\it hyperbolic stable surface} is a stable surface with a hyperbolic structure on each part.

 The augmented moduli space of genus $g$ is the space $\widehat{\M}_g$ of hyperbolic  stable surfaces of genus $g$ up to isometry. This space provides a compactification of moduli space, once given a topology that intuitively works as follows. Consider a point  $X\in \widehat{\M}_g$ with nodes $N_1,\dots, N_r$.  A sequence $X_n\in \M_g$ converges to   $X\in \widehat{\M}_g$  if there is a family of geodesics $\mathcal{F}_n=\{\a_1^n,\dots, \a_r^n\}$ in $X_n$ so that:
\begin{itemize}
\item[(i)] the stable surfaces $X_n/\mathcal{F}_n$ are homeomorphic to $X$;
\item[(ii)]    the length of each $\a_i^n, i=1,\dots, r,\,$ tends to zero when $n\to \infty$; and 
\item[(iii)]  away from the curves $\a_i$ of $X_n$ and from the nodes of $X$, the hyperbolic surfaces $X_n$ are close to $X$ when $n\to \infty$.
 \end{itemize}
A way of formalizing this is by first considering marked surfaces and the augmented  Teichm\"uller space. See, for instance, \cite{Hubbard} for a detailed exposition on this.

 The augmented moduli space is stratified in the following way. For each stable graph  $\mathcal G$, we consider the stratum $\mathcal E(\mathcal G)$, defined as  the space   of stable hyperbolic surfaces $X$ whose associated graph $\G_X$ is isomorphic to ${\mathcal  G}$. This space is homeomorphic to  the product of the moduli spaces of the parts of $X$, which are punctured surfaces. 
Since there is a finite number of non-isomorphic  stable graphs of a fixed genus $g$, then $\widehat{   {\mathcal M}_g}$ decomposes into a finite number of strata, one for each stable graph. 
  This is called the {\it topological stratification} of $\widehat{\M}_g$.  
  
  Certainly, we can consider the closure of an equisymmetric stratum $\M_g(G,\O,\Phi)$ in $\widehat{\M_g}$, and its boundary, which is the center of attention of this paper.

\section{Statement of main theorem in \cite{diaz1}}\label{sec:MainThmInDiGo}

Consider a topological action on a surface $S $ given by a (surface kernel) epimorphism  $\Phi\colon \pi_1(\O,*)\to G$. Let $\S$ be an admissible multicurve in $\O$ and let $\tilde \S=p^{-1}(\S)$ be its preimage under the branched covering.   
The main theorem in \cite{diaz1} gives a criterion to determine the dual graph of $\tilde \S$ from  $\S$ and the epimorphism $\Phi$. We recall here that theorem. 

There are some new concepts and notation needed in the theorem.
For each suborbifold $\O_j$ of $\O\setminus \S$, we consider the restriction $\Phi_j\colon \pi_1(\O_j)\to G$. For each $\g\in \S$, we consider  similarly  the restriction $\Phi_{\g}\colon \pi_1(\g)\to G$, and  we also consider a dual curve $c_{\g}$. We will give details of these concepts after the statement of the theorem. For technical reasons in order to define the dual curves, we give an orientation to the graph $\D_{\S}$.
Finally, $\S^1$ denotes  the set of loops or semiedges of $\D_{\S}$,  $\S^2$ denotes the set of remaining edges  and $\S_j$ denotes the set of edges and semiedges adjacent to the vertex $\O_j$. We warn that, in order to slightly simplify the reading, we have adopted here the term  ``dual curves" which was not explicit in the original statement in \cite{diaz1}. 

\begin{thm}
\label{thm:MainDiaz1}
Consider a topological action on a surface   given by an  epimorphism  $\Phi\colon \pi_1(\O,*)\to G$ and let $\S$ be an admissible multicurve in $\O$. Then 
 the dual graph $\D_{\tilde \S}$ of the multicurve $\tilde \S=p^{-1}(\S)$ and a map     $p_*\colon \D_{\tilde{\S}}\to \D_{\S}$  are determined as follows.

\begin{enumerate}

\item The preimages under $p_*$ of a vertex $\O_j$ of $\D_{\S}$ (i.e., a component of $\O\setminus\S$) are denoted   $V_{j,g\Im\Phi_j}, g\in G$; thus the vertex $ \O_j$ has $\frac{|G|}{|\Im\Phi_j|}$ preimages.
\item The preimages of an edge $\g$ are denoted  $E_{\g,g\Im\Phi_{\g}}, g\in G$; thus the edge $\g$ has  $\frac{|G|}{|\Im\Phi_{\g}|}$ preimages.
\item    (Edges connecting vertices.) If $\g$ goes from $\O_{j_1}$ to $\O_{j_2}$, then the edge $E_{\g,g\Im\Phi_{\g}}$   joins the vertices $V_{j_1, g\Im\Phi_{j_1}}$ and  $V_{j_2, g{\Phi(c_{\g})}\Im\Phi_{j_2}}$.
 \item (Degrees of    vertices.) The degree of the vertex $V_{j,g\Im\Phi_j}$ 
is equal to 
 $$
 D_{j}= |\Im \Phi_j| \left( \sum_{\g\in \S_j \cap  \S^2 }\frac{1}{|\Im\Phi_{\g}|} +2\sum_{\g\in\S_j\cap  {\S^1}}\frac{1}{|\Im\Phi_{\g}|} \right).
  $$
 \item   (Weights of  vertices.) The weight $w^j$
 of the vertex $V_{j,g\Im\Phi_j }$ is equal to   
$$\quad \quad
w^j
=1 
 -\frac{1 }{2}\left(|\Im \Phi_j|\, \chi(\O_j)+ D_j\right).  
$$ 
 
\end{enumerate}

\end{thm}

{\bf Remarks.} (a) The theorem gives an action of $G$ on $\DArriba$ whose quotient map is $p_*$. 
Indeed, we can define
$$
g(V_{j,h\Im\Phi_j})=  V_{j,gh\Im\Phi_j}, \; g(E_{\g,h\Im\Phi_j})=  V_{\g,gh\Im\Phi_j}
$$
Notice that if  $E_{\g,h\Im\Phi_j}$ joins $V_{j_1,h\Im\Phi_{j_1}}$ and $V_{j_2,h\Phi(c_{\g})\Im\Phi_{j_2}}$, then  $g(E_{\g,h\Im\Phi_j})=E_{\g,gh\Im\Phi_j}$ joins $V_{j_1,gh\Im\Phi_{j_1}}=g(V_{j_1, h\Im\Phi_{j_1}})$ and $V_{j_2,g\Phi(c_{\g})h\Im\Phi_{j_2}}=g(V_{j_2,h\Phi(c_{\g})\Im\Phi_{j_2}})$. 

(b) The items (4) and (5) say that all the preimages by $p_*$ of a vertex of $\DAbajo$ have the same degree and the same weight. 
 
\noindent
{\bf Explanations.} In order to define the restrictions $\Phi_j,\Phi_{\g}$ we need to carefully choose base points and auxiliary paths, namely:
\begin{enumerate}

\item choose a base point $*_j$ in each suborbifold $\O_j$ (for simplicity, choose $*_1$ equal to the initial base point $*$) and  a base point $*_{\g}$ in each curve $\g\in \S$;
\item choose  simple  paths $\b_{j,\g}$ contained in $\O_j$ going from $*_j$ to $*_{\g}$ for all the edges $\g$ adjacent to $\O_j$. If $\g$ is a loop or semiedge of $\mathcal D_{\S}$, we choose two paths $\b_{j,\g}^a,\b_{j,\g}^b $  
 from $*_j$ to $*_{\g}$ such that $\b_{j,\g}^a(\b_{j,\g}^b)^{-1}$ intersects   $\g$ exactly once and,  in the case that $\g$ is a semiedge (arc in $\O_j$ joining two cone points of order 2), $\b_{j,\g}^a(\b_{j,\g}^b)^{-1}$ bounds a disc which contains just one of the endpoints of $\g$  and no other cone point.
   Moreover we choose all these paths   so that  they are disjoint except at their endpoints.
 
\end{enumerate}
 
 We finally choose an spanning tree $\T$ of $\D_{\S}$. All these choices allow to consider paths $\b_j$ going from $*$ to $*_j$ following the unique path in $\T$ from $\O_1$ to $\O_j$ and using the auxiliary paths $\b_{j,\g}$. Now we can define the restriction $\Phi_j$ as  $\Phi_j(\a)=\Phi(\b_j\a\b_j^{-1})$. 
 Similarly, we can consider paths $\b_{\g}$ for each $\g\in\S$ and define the restrictions $\Phi_{\g}$ to the fundamental group of the curve $\g$. We remark that if $\g$ is an arc, it is considered as a 1-orbifold with its endpoints cone points of order 2. Thus,  its fundamental group is $\Z_2*\Z_2$ generated by loops around the cone points. 
See \cite{diaz1} for more details. 
 
 Finally, the dual curves are as follows (see Figure \ref{fig:CaminosCGamma}): 
 \begin{enumerate}
 \item if the oriented edge  $\g$ of $\DAbajo$  goes from $\O_{j_1}$ to $\O_{j_2}$, different vertices, then $c_{\g}= \b_{j_1}\,\b_{j_1,\g}\, \b_{j_2,\g}^{-1}\,\b_{j_2}^{-1} $;
 \item  if $\g$ is a loop or a semiedge adjacent to the vertex $\O_j$, then $c_{\g}=  \b_{j}\, \b_{j,\g}^a(\b_{j,\g}^b)^{-1}\b_j^{-1}$. Notice that, if $\g$ is a semiedge, $\Phi(c_{\g})$ is one of the generators of $\Im\Phi_{\g}$. 
 
 \end{enumerate}

  \psfrag{*}{$*$}

\psfrag{*j1}{$*_{j_1}$}
\psfrag{*j2}{$*_{j_2}$}

\psfrag{*g}{$*_{\gamma}$}
\psfrag{*j}{$*_{j}$}
\psfrag{bj1}{$\beta_{j_1}$}
\psfrag{bj2}{$\beta_{j_2}$}
\psfrag{bg}{$\beta_{\gamma}$}
\psfrag{bj1g}{$\beta_{j_1,\gamma}$}
\psfrag{bj}{$\beta_{j}$}
\psfrag{bjga}{$\beta_{j,\gamma}^a$}
\psfrag{bjgb}{$\beta_{j,\gamma}^b$}

\psfrag{bj2g}{$\beta_{j_2,\gamma}$}

\psfrag{Oj1}{$\bf{\mathcal{O}_{j_1}}$}
\psfrag{Oj2}{$\bf{\mathcal{O}_{j_2}}$}
 \psfrag{D}{$\Delta$}
\psfrag{g}{\textcolor{red}{ $ \gamma$} }
\psfrag{d}{$ {\delta}$}

\psfrag{Oj}{$\bf{\mathcal{O}_{j}}$}
\psfrag{cg}{\textcolor{blue}{$c_{\gamma}$}}
  \begin{figure}
\center
\includegraphics[height=5.5cm,width=15.5cm]{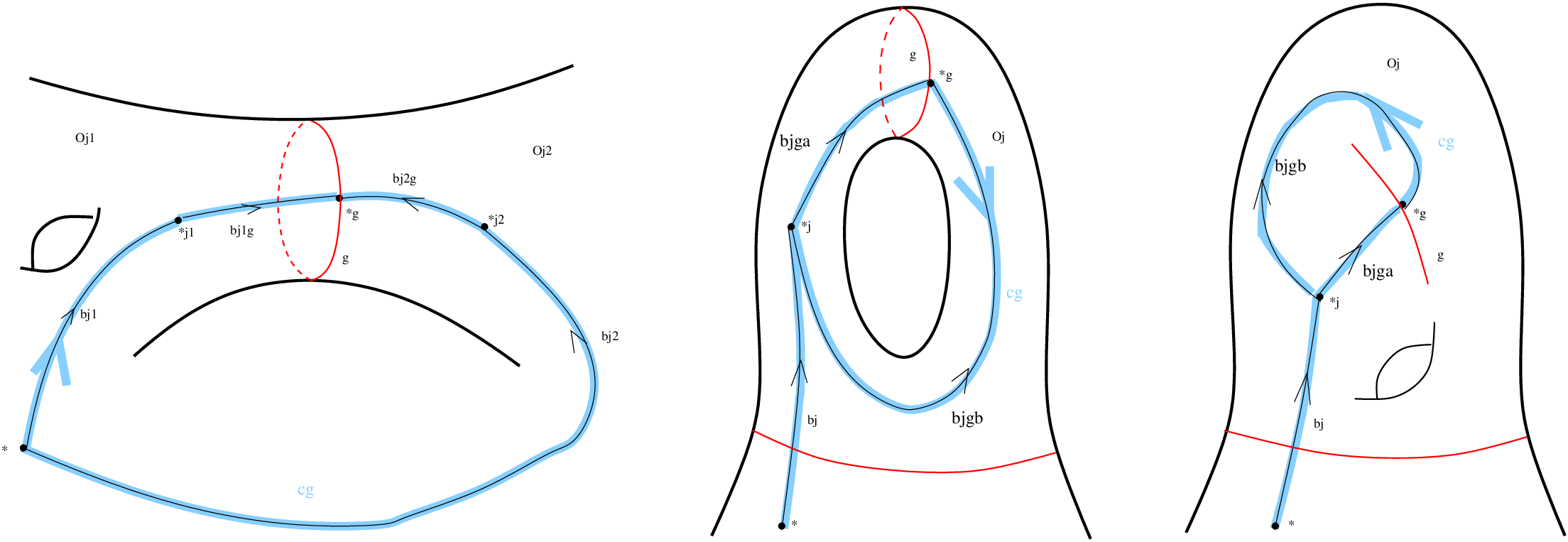}
\caption{Dual curves $c_{\gamma}$ 
}
\label{fig:CaminosCGamma}
\end{figure}

 \noindent
 {\bf Notation.} To the weighted  dual graph $\D_{\S}$ of a multicurve   $\S$ we can add the information given by the action: to  each vertex $\O_j$ we add the label $\Im \Phi_j$ and to each edge or semiedge $\g$ we add the labels  $\Im\Phi_{\g}$ and $\Phi(c_{\g})$.  Knowing this graph allows to obtain all the information (1)-(5) of Theorem \ref{thm:MainDiaz1}. We will call this graph the {\it extra decorated graph} associated to $\S$ (with respect to the action $\Phi$).

\section{First simplifications}\label{sec:FirstSimplifications}

Theorem \ref{thm:MainDiaz1} admits two important simplifications in the cases where the orbifold $\O$ has genus 0 and no cone points of order 2 and in the case where the group $G$ is abelian.

\subsection{Genus of $\O$   equals zero and   no cone points of order 2 }\label{sec:genusO=0}

Since the genus of $\O$ is zero, the dual graph $\D_{\S}$ of any multicurve $\S$ on $\S$ is a tree. Moreover, since there are no cone points of order 2, this tree has no semiedges. 
  As a consequence, all the dual curves $c_{\g}$ are homotopically trivial and $\Phi(c_{\g})$ is the identity. Item (3) in Theorem \ref{thm:MainDiaz1} thus says that if $\g$ is an edge of $\D_{\S}$ going 
from $\O_{j_1}$ to $\O_{j_2}$, then for each $g\in G$ the edge $E_{\g,g\Im\Phi_{\g}}$ joins the vertices $V_{j_1,g\Im\Phi_{j_1}}$ and $V_{j_2,g\Im\Phi_{j_2}}$. Then we can interpret $\D_{\tilde\S}$ as $(\sqcup_{g\in G}  \mathcal D_{\S}^g)/\sim$, where  $\mathcal D_{\S}^g$ is a copy of $\D_{\S}$ and  where two vertices
$V_{j,g\Im\Phi_j}, V_{j,g'\Im\Phi_j}$ are identified if and only if $g\Im\Phi_j=g'\Im\Phi_{j}$ and similarly with the edges, by (1) and (2) of that theorem. Thus we have:

\begin{prop} \label{prop:GenusO=0}
Consider a topological action given by an epimorphism  $\Phi\colon \pi_1(\O,*)\to G$ where the signature of $\O$ is equal to $(0; m_1,\dots, m_k)$, $m_i\geq3$. Let $\S$ be an admissible multicurve in $\O$. Then   $\D_{\tilde\S}=(\sqcup_{g\in G}  \mathcal D_{\S}^g)/\sim$, where  $\mathcal D_{\S}^g$ is a copy of $\D_{\S}$, 
$V_{j,g\Im\Phi_j}\sim  V_{j,g'\Im\Phi_j}$   if and only if $g(g')^{-1}\in \Im\Phi_j $ and $E_{\g,g\Im\Phi_{\g}}\sim  E_{\g,g'\Im\Phi_{\g}}$   if and only if $g(g')^{-1}\in\Im\Phi_{\g}$.
\end{prop}

\subsection{Abelian case reduces to homology}\label{sec:Simplification-Abelian}

Suppose that $G$ is abelian. Then the homomorphism $\Phi\colon\pi_1(\O,*)\to G$ factors through the abelianized group  $H_1(\O)$ of $ \pi_1(\O,*)$.
By abuse of notation, we will still call $\Phi$ to the homomorphism from the homology group. 
 If the signature of $\O$ is $(\s;m_1,\dots, m_k)$, then $H_1(\O)$ is generated by $\a_i,\b_i, \delta_l$, $1\leq i\leq \sigma, \, 1\leq l\leq k$, where the $\a_i,\b_i$ are a homology basis of the underline space of  $\O$ and the $\d_l$ are meridians about the cone points. This is a great simplification, since we do not have to take care of the base points, in particular we do not need the choices of any base point, nor the choices of the paths $\b_{j,\g}$.  
The restrictions $\Phi_j$ and $\Phi_{\g}$ are easily computed: it is needed to express the homology  generators of $\O_j $ or $\g$ in terms of the homology generators of $\O$. For example, if $\O_j$ has genus zero, $\Phi_j$ only depends on the meridians about the cone points contained in $\O_j$. 

The homology classes of the dual curves $c_{\g}$ are easily chosen when $\g$ is contained in the boundary of just one suborbifold $\O_j$: if $\g$ is an arc, $c_{\g}$ is a meridian about one of  if its endpoints, and if $\g$ is a simple closed curve, $c_{\g}$ is  a simple closed curve in the closure of $\O_j$ intersecting $\g$ exactly one.   
If $\g$ is in the boundary of two different suborbifolds, and more precisely, if  $\g$ is an oriented edge in $\D_{\S}$ from $\O_{j_1}$ to $\O_{j_2}$, consider the oriented  closed path $\omega$  in $\D_{\S}$ formed by the oriented edge $\g$ followed by the path in the spanning tree from $\O_{j_2}$ to $\O_{j_1}$. Now, take a simple oriented closed curve $c_{\g}$ in $\O$ traveling through the  suborbifolds indicated by   $\omega$.  

To appreciate these simplifications, let us compare with an example of  a non-abelian action. 

\noindent
{\bf Example.} The pyramidal action is a certain  action of the dihedral group $D_n$ on a surface $S_n$  with quotient the orbifold $\O$ with signature  $(0;2,2,2,2,n)$.
In   \cite{diaz1} we considered  examples of multicurves $\S =\{\g\}$ and $\Sigma'=\{\g'\}$ where  both $\g,\g'$ are arcs joining the same cone points of $\O$ although not homotopic. It was shown that the subgroup  $\Im\Phi_{\g}$ has order 2, while $\Im\Phi_{\g'}$ is the whole group if $n$ is odd, and has order $n$ if $n$ is even. For $n>2$ this implies that the number of edges of $\DArriba$ and $\D_{\widetilde{\S'}}$  are different, so they are different graphs and provide different topological strata in the boundary.

This phenomenon occurs because the dihedral group is not abelian. Indeed, if we have any abelian action on $S_n$ with  quotient an orbifold of signature $(0;2,2,2,2,n)$, and we consider the multicurves $\S =\{\g\}$ and $\Sigma'=\{\g'\}$ as before, then $\O\setminus\g$ and $\O\setminus\g'$ have the same homology generators, and so do $\g$ and $\g'$, and we can choose the dual curves $c_{\g},c_{\g'}$ to be equal. Thus, the graphs $\D_{\S}, \D_{\S'}$ are equal and provide the same topological stratum in the boundary.

\section{Cyclic $p$-gonal locus, $p\geq 3$}\label{sec:p-gonal}
\noindent{\bf Definition.} A {\it cyclic $p$-gonal action}, for $p$ a prime number, is an action of $\Z_p$ on a surface $S_{\!\!g}$ of genus $g$ whose quotient orbifold $\O$  has signature $(0; p,\dots^M, p)$.
Thus, when $p\geq 3$, these actions satisfy both simplifications explained in Section \ref{sec:FirstSimplifications}, so we will study this case in this section.

The homology group of    $\O$ with the above signature is generated by oriented  loops $x_1,\dots,$ $ x_M$ around the cone points $q_1,\dots, q_M$ of $\O$. We take all these loops with the same orientation so that its product $x_1\dots x_M$ is trivial.  If $a$ is the generator of $\Z_p$, the epimorphism $\Phi\colon H_1(\O)\to G$ with   $\Phi(x_i)=a^{m_i}$ is denoted briefly by its {\it generating vector}  $(a^{m_1},\dots, a^{m_M})$. We say that $m_i$ is the {\it exponent} of the cone point $q_i$.

 The Riemann-Hurwitz formula and the Riemann existence theorem for topological actions permit obtain all  the $p$-gonal actions, as follows. 

 \begin{prop}\label{prop:Epi-p-gonal}
 The  epimorphism $\Phi\colon H_1(\O)\to \Z_p$ with generating vector $(a^{m_1},\dots, a^{m_M})$ determines a topological action on a surface $S$ with quotient an orbifold with signature $(0; p,\dots^M, p)$ if and only if $m_i\in\{ 1,\dots, p-1\}$ for all $i $ and    $ \sum_{i=1}^M m_i$ is multiple of $p$. The genus of the surface $S$ is  and  $g= \frac{ (p-1)(M-2)}{2}$.
\end{prop}
\noindent{\bf Remark.} The above proposition does not say anything about the topological equivalence of the actions. See \cite{Broughton-Classifying} and \cite{kimura} to find all the topological inequivalent $p$-gonal actions on a surface of genus 2, 3 and 4. 
For instance, in genus 4 there are two inequivalent 3-gonal actions and three inequivalent 5-gonal actions.

Let $\S$ be a multicurve in $\O$. 
Since $\Z_p$ has only the two trivial subgroups, items (1) and (2) of  Theorem \ref{thm:MainDiaz1} say that the preimage of any vertex  of $\DAbajo$ is either 1 vertex or $p$ vertices, and the same for any edge of $\DAbajo$. We will say that  a vertex $\O_j$ (resp. an  edge $\g$) of $\DAbajo$ is {\it multiple of $p$} if its preimage is a set of $p$ vertices (resp.  edges) of $\DArriba$  or, equivalently,  if $\Im\Phi_j=\{0\}$ (resp. $\Im\Phi_{\g}=\{0\}$).

 It is easy to see whether a vertex or an edge is multiple of $p$, this is done in the following lemma. 

\begin{lem} \label{lem:p-gonal}
Let $\Phi\colon H_1(\O)\to \Z_p$ be a cyclic $p$-gonal action with generating vector $(a^{m_1},\dots,$ $ a^{m_M})$. Let $\S\subset \O$ be an admissible multicurve.
Then
\begin{itemize}
\item[(i)] An edge $\g$ of $\D_{\S}$ is multiple of $p$ if
and only if the sum of the exponents of the cone points contained in a disc bounded by $\g$ is multiple of $p$. 
\item[(ii)] A vertex    $\O_j$ of $\D_{\S}$ is multiple of $p$
 if and only if the suborbifold $\O_j$ does not contain any cone point and all the edges $\g$ incident to the vertex $\O_j$ are multiple of $p$. In other words, if $c_j+n_j=0$, where 
  $c_j$ is the number of cone points contained in $\O_j$
  and $n_j$ is the number of edges  not multiple of $p$ that are incident to $\O_j$.

\item[(iii)] If $\O_j$ is a vertex not multiple of $p$ then $c_j+n_j\geq 2$. 
\end{itemize}
\end{lem}

\begin{proof}
(i) According to Theorem \ref{thm:MainDiaz1}(2), the edge $\g$ has $|G/\Im\Phi_{\g}|$ preimages. Now, $\Phi_{\g}$ is the restriction of $\Phi$ to the homology group of $\g$ (since $G=\Z_p$ is abelian).
Suppose the curve $\g$ bounds a disc with containing the  cone points  $q_{i_1},\dots, q_{i_k}$. Since the loops $x_i$ have been chosen with the same orientation, the generator of the homology of $\g$ is $x_{i_1}   \dots   x_{i_k}$. Therefore, 
 $\Im\Phi_{\g}$ is generated by $a^{m_{i_1}+\dots +m_{i_k}}$. If ${m_{i_1}+\dots +m_{i_k}}$ is multiple of $p$ then  $\Im\Phi_{\g}$ is trivial, and then $\g$ lifts to $p$ edges. Otherwise,   $\Im\Phi_{\g}$ is the total $G$  and  $\g$ lifts to just one edge.

(ii) According to Theorem \ref{thm:MainDiaz1}(1) and in a similar way as before, we have to determine when the subgroup $\Im\Phi_{j}$ is trivial or the total group. Because $\O_j$ has genus 0, the subgroup $H_1(O_j)$ is generated by loops around the cone-points contained in $\O_j$ and by the boundary components of $\O_j$, which correspond to the edges of $\D_{\S}$ incident to the vertex $\O_j$. Therefore, $\Im\Phi_{j}$ is trivial if and only if  the condition written in the statement holds.

(iii) By the previous item,  $n_j+c_j\not=0$. Suppose that  $n_j=1$.
  By item (i) and since $n_j=1$, the sum of exponents of the cone points  outside $\O_j$ is not multiple of $p$. Since   the sum of all the exponents is multiple of $p$ (Proposition \ref{prop:Epi-p-gonal}), there must be some cone point in $\O_j$, thus $c_j\geq 1$. 
Finally, $n_j=0,c_j=1$ also contradicts Proposition \ref{prop:Epi-p-gonal} since no exponent is multiple of $p$. 
\end{proof}

We can now completely describe the graph $\DArriba $ for any multicurve $\S$ in $\O$, therefore the strata in the boundary of the $p$-gonal equisymmetric stratum.

\begin{thm}\label{thm:p-gonal}
Let $\Phi\colon H_1(\O ,*)\to \Z_p$ be a cyclic $p$-gonal action. Let $\S\subset \O$ be an admissible multicurve and let $\DAbajo$ its dual graph. Then the dual graph $\DArriba$ of $\,\tilde\S$ is obtained as $(\sqcup_{t=1}^p\D^t)/\hspace{-0.1cm}\sim$ where $\D^t$ are copies of $\D_{\S}$ and $\sim$ is as follows. We denote $V^t_j$  the vertex of $\D^t$ corresponding to the vertex $V_j$ of $\D_{\S}$, and similarly with the edges. Then:

\begin{enumerate}

\item  $V^t_j\sim V^s_j$ if and only if   $V_j$ is  not multiple of $p$ .

\item Similarly, $E^t_j\sim E^s_j$ if and only if $E_j$ is   not multiple of $p$.

\item If a vertex $V=\O_j$  is multiple of $p$, then the degree $D_j$ of any of its preimages is equal to  the degree of $V$ and its weight $w^j$ is equal to zero.

\item  If a vertex $V=\O_j$ is not multiple of $p$ then the  degree of any of its preimages is equal to $D_j=p\,m_j+n_j$, where $m_j$ is the number of edges multiple of $p$ that are incident to $\O_j$  and $n_j$ is the number of edges  not multiple of $p$ that are incident to $\O_j$.  Its weight is $w^j =\frac{p-1}{2} (n_j+c_j-2)  $, where $c_j$ is the number of cone points contained in $\O_j$.
\end{enumerate} 
\end{thm}

We remark that the weight $w^j$ obtained in (iv) is non-negative by Lemma \ref{lem:p-gonal}(iii).

\begin{proof}
Since $\O $ has genus 0, the dual graph $\DArriba $ is quotient of  the union of $p$ copies of $\DAbajo$, by Proposition \ref{prop:GenusO=0}. The precise identifications given by $\sim$ come from that proposition  and from Lemma \ref{lem:p-gonal}. Thus we have (1) and (2).

(3)   If $V=\O_j$ is multiple of $p$, then $\O_j$ does not contain any cone point and all the edges incident to $V$ are multiple of $p$, by Lemma \ref{lem:p-gonal}. Thus,   $|\Im\Phi_j|=1$ and $|\Im\Phi_{\g}|=1$  for all $\g$ in the boundary of $\O_j$. Then  Theorem \ref{thm:MainDiaz1}(4) gives that $D_j$ is equal to the degree of $ \O_j$,   and item (5) of that theorem 
says that 
$$w^j=1-\frac{1}{2}(\chi(\O_j)+{\rm deg}(\O_j))= 0
$$
because $\O_j$ has genus 0 and no cone points, so $\chi(\O_j)+{\rm deg}(\O_j)=2$.

(4) If $\O_j$ is not multiple of $p$, then $|\Im\Phi_j|=p$. Denote by $\mathcal M_j$ the set of (multiple of $p$)-edges adjacent to $\O_j$ and by $\mathcal N_j$ the set of (non-multiple of $p$)-edges adjacent to $\O_j$. Then, 
by Theorem \ref{thm:MainDiaz1}(4), we have 
$$
D_j= |\Im \Phi_j| \left( \sum_{\g\in \S_j \cap  \S^2 }\frac{1}{|\Im\Phi_{\g}|}\right)=p(\sum_{\g\in \mathcal M_j} \frac{1}{1}+ \sum_{\g\in\mathcal N_j}\frac{1}{p} )=p|\mathcal M_j|+|\mathcal N_j|=p\,m_j+n_j
$$
For the weight, since  $\O_j$ has genus 0, $c_j$ cone points of order $p$ and   degree  $m_j+n_j$, its Euler characteristic is  $
\chi(\O_j)= 2- n_j-m_j-c_j\left(1-\frac{1}{p} \right).
$ Thus, by 
Theorem \ref{thm:MainDiaz1}(5), we have
$$
w^j=1-\frac{1}{2}\left( p \Big(2- n_j-m_j-c_j \frac{p-1}{p}  \Big) +  p\,m_j+n_j \right) = \frac{p-1}{2} (n_j+c_j-2)
$$
\end{proof}

\subsection{Examples}  From the lists of topological actions in genus up to 4, 
we extract the {\it non-rigid} cyclic  $p$-gonal actions, that is, those cyclic $p$-gonal actions whose equisymmetric strata   is not a single point. For each of them, we find all the topological strata at the boundary. 

\subsubsection{Non-rigid $p$-gonal  actions in genus $g=2$}
  According to \cite{Broughton-Classifying}, on a surface of genus 2 there is one cyclic 3-gonal action and one cyclic 5-gonal action, which is rigid, and no other $p$-gonal actions.  The 3-gonal action is given by the epimorphism $\Phi\colon H_1(\O)\to \Z_3$ such that  $\O$ has signature $(0;3,3,3,3)$ and $\Phi$ has generating vector  $(a,a,a^{2},a^{2})$. An admissible multicurve on  an  orbifold of genus 0 and with  4 cone points consists just on one curve separating two cone points from the other two.  Taking into account the exponents of generating vector, we have two cases for $\S$, shown in  figure \ref{fig:3-gonal-g=2}. All the suborbifolds involved contain cone points, so none of them is  multiple of 3. Thus each vertex of $\DAbajo$ lifts to a vertex in $\DArriba$.  In the first case, the sum of exponents in a disc bounded by $\g$ is not multiple of 3, thus $\g$ is not multiple of 3 so it lifts to a single edge in $\DArriba$. In the second case $\g$ is multiple of 3 and so lifts to three edges. 

\psfrag{a}{\tiny{\textcolor{red}{$1$}}}
\psfrag{a1}{\tiny \textcolor{red}{$  2$}}
\psfrag{2}{\tiny \textcolor{red}{$\!\!\!11$}}
\psfrag{0}{\tiny \textcolor{red}{$\!\!\!12$}}
\psfrag{-2}{\tiny \textcolor{red}{$\!\!\!22$}}
\psfrag{z}{\tiny \textcolor{blue}{$0$}}
\psfrag{1}{\tiny \textcolor{blue}{$1$}}
\psfrag{G}{\tiny \textcolor{blue}{$G$}}
\psfrag{id}{\tiny \textcolor{blue}{$\{id\}$}}
\psfrag{DAb}{\tiny \textcolor{blue}{$\DAbajo$}}
\psfrag{DAr}{\tiny \textcolor{blue}{$\DArriba$}}
\hspace{1truecm}

\begin{figure}
\includegraphics[height=3.5cm,width=7cm]{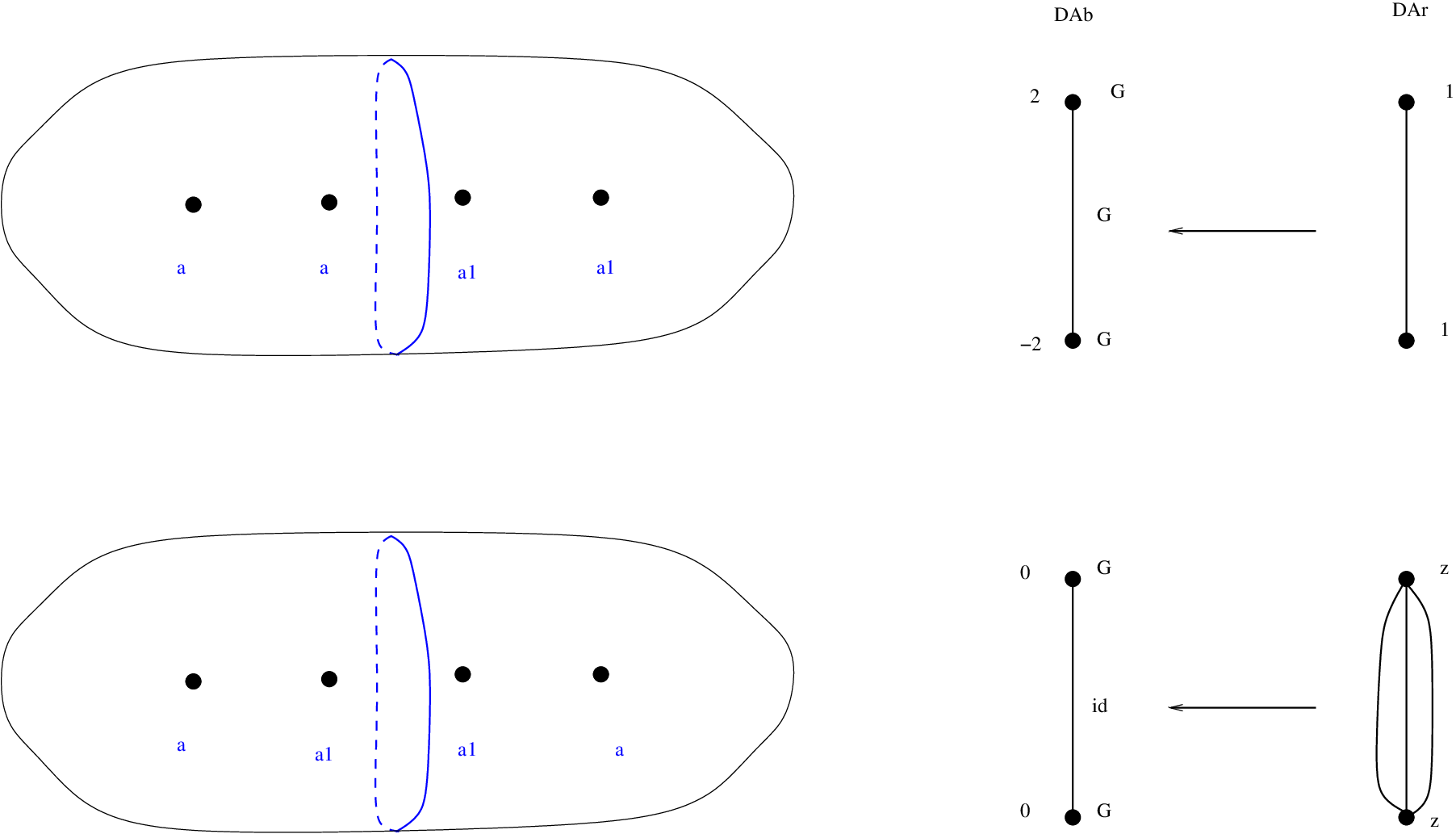}
\caption{Cyclic 3-gonal action in genus 2.  
The red numbers, both in $\O$ and in  $\DAbajo$, indicate exponents of the generating vector. The subgroups $\Im\Phi_j$ and $\Im\Phi_{\g}$ are  also written at $\DAbajo$. } 
\label{fig:3-gonal-g=2}
\end{figure}

\subsubsection{Non-rigid $p$-gonal  actions in genus $g=3$} There is only one non-rigid $p$-gonal action, for $p=3$, signature of $\O$ equal to $(0;3^5)$ and with generating vector $(a,a,a,a,a^{2})$. 
 Now a multicurve may have 1 or 2 components. Figure 
  \ref{fig:p-gonal-g=3} shows the four possible  situations.  We conclude that the boundary of the 3-gonal  equisymmetric  locus in $\M_3$ intersects 4 different topological strata. 
 
 \psfrag{a}{\tiny{\textcolor{red}{$1$}}}
\psfrag{a1}{\tiny \textcolor{red}{$2$}}
\psfrag{2}{\tiny \textcolor{red}{$\!\!\!11$}}
\psfrag{3}{\tiny \textcolor{red}{$\!\!\!\!111$}}
\psfrag{0}{\tiny \textcolor{red}{$\!\!\!12$}}
\psfrag{-2}{\tiny \textcolor{red}{$1$}}
\psfrag{-1}{\tiny \textcolor{red}{$2$}}
\psfrag{z}{\tiny \textcolor{blue}{$0$}}
\psfrag{1b}{\tiny \textcolor{blue}{$1$}}
\psfrag{2b}{\tiny \textcolor{blue}{$2$}}
\psfrag{1}{\tiny \textcolor{red}{$\!\!\!\!112$}}
\psfrag{G}{\tiny \textcolor{blue}{$G$}}
\psfrag{id}{\tiny \textcolor{blue}{$id$}}
\psfrag{DAb}{\tiny \textcolor{blue}{$\DAbajo$}}
\psfrag{DAr}{\tiny \textcolor{blue}{$\DArriba$}}
\hspace{1.5truecm}
\begin{figure}
\includegraphics[height=8cm,width=8cm]{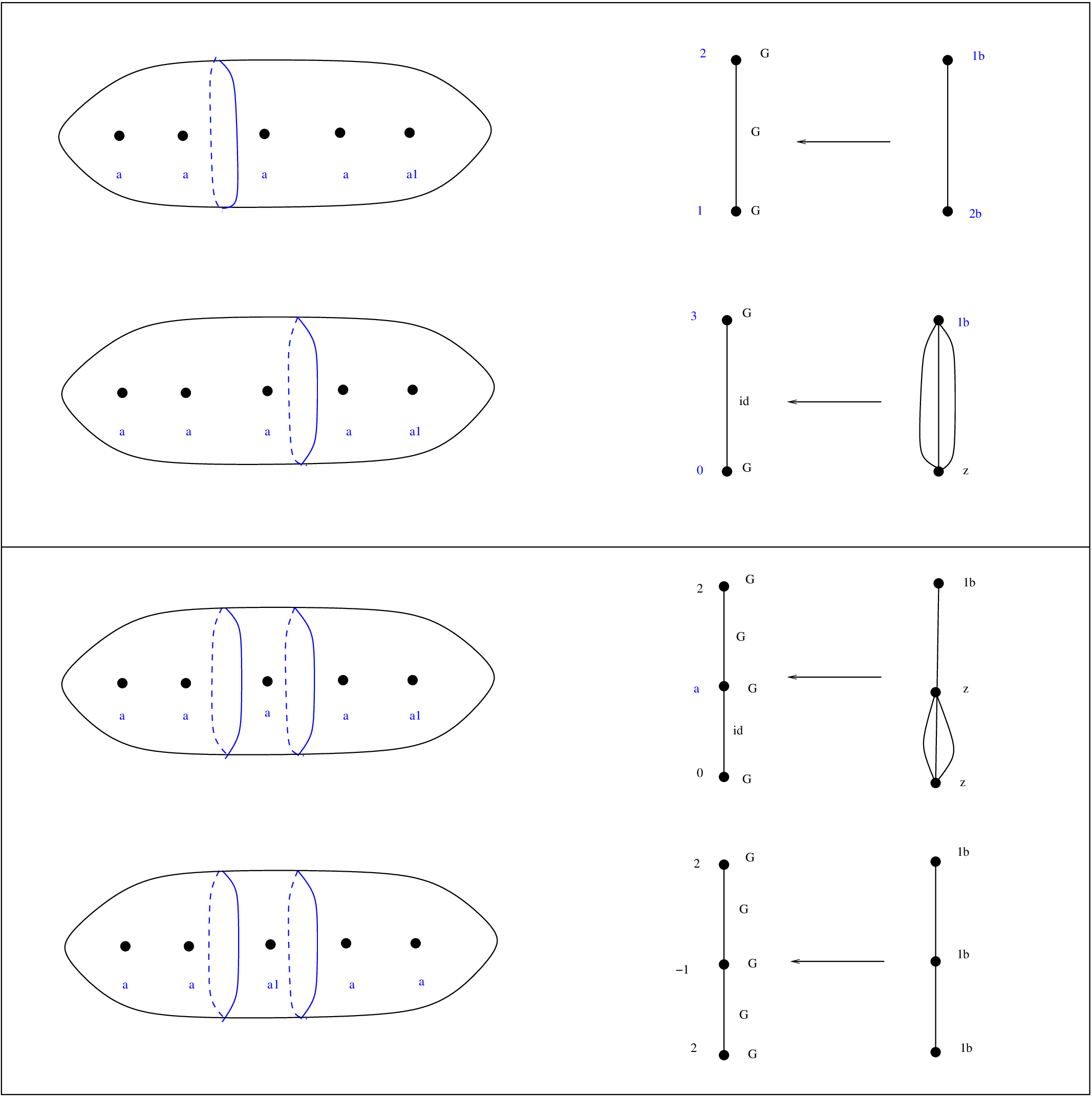}
\caption{Cyclic 3-gonal action in genus 3.  The boundary of the  equisymmetric locus intersects 4 topological strata. }
\label{fig:p-gonal-g=3}
\end{figure}

 \subsubsection{Non-rigid $p$-gonal  actions in genus $g=4$} Using  Kimura notation, there are the following non-rigid $p$-gonal actions in genus $g=4$.
 \begin{itemize}
\item $\Phi_7\colon H_1(0; 3^6)\to \Z_3$,\quad \quad  $(a,a,a,a,a,a)$

\item $\Phi_8\colon H_1(0; 3^6)\to \Z_3$,\quad  \quad $(a,a,a,a^2,a^2,a^2)$

\item $\Phi_{12}\colon H_1(0; 5,5,5,5)\to \Z_5$, \quad $(a,a,a,a^2)$

\item $\Phi_{13}\colon H_1(0; 5,5,5,5)\to \Z_5$, \quad $(a,a,a^4,a^4)$

\item $\Phi_{14}\colon H_1(0; 5,5,5,5)\to \Z_5$, \quad $(a,a^2,a^3,a^4)$

\end{itemize}

Table \ref{Tabla:Phi_7} shows all possible topological strata at the boundary of the equisymmetric locus determined by the action $\Phi_7$. Tables  \ref{Tabla:Phi_8(I)} and  \ref{Tabla:Phi_8(II)} do the same for  the equisymmetric locus determined by $\Phi_8$. 
In both actions the orbifold $\O$ is a sphere with 6 cone points. There are four possible admissible multicurves $\S$, two of them pants decompositions (maximal number of components). Pants decompositions   give rise to minimal strata in the boundary of the equisymmetric space, in the sense that their closure do not contain any other strata.  
 Table \ref{Tabla:Phi_8(I)} shows all the non-minimal topological strata and Table \ref{Tabla:Phi_8(II)} shows the minimal strata.  Finally, Table \ref{Tabla:Phi12Phi13Phi14} shows the same for the actions $\Phi_{12}$, $\Phi_{13}$ and $\Phi_{14}$. 

In all these tables, the graphs $\DAbajo$ are shown at the left   and the graphs $\DArriba$ are shown at the rightmost column. The numbers assigned to the vertices of $\DAbajo$ indicate the exponents of the generating vector.

\begin{table}
\begin{center}
\begin{tabular}{|c|ccc|}
\hline 
\multirow{6}{*}{ \includegraphics[scale=0.2]{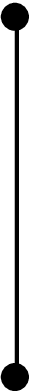}}  
& 
\multirow{3}{*}{ \includegraphics[scale=0.15]{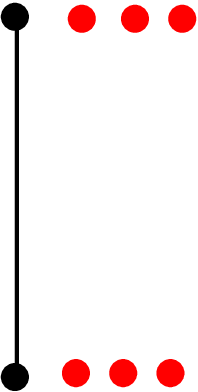}} 
 &   & 
 \multirow{3}{*}{ 
 \psfrag{a}{\tiny 1}
 \psfrag{b}{\tiny 1}
 \includegraphics[scale=0.15]{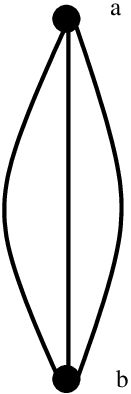}
 }  \\ 
%\hline 
  &   & \hspace{0.3cm} $\longleftarrow$ %\hspace{0.3cm}
    &  \\ 
%\hline 
  &   &   &   \\ 
\cline{2-4}
  & \multirow{3}{*}{ \includegraphics[scale=0.15]{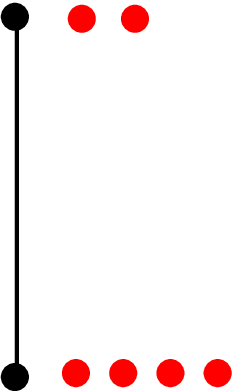}} 
  &   & \multirow{3}{*}{ 
  \psfrag{a}{\vspace{0.1cm}\tiny 1}
  \psfrag{b}{\tiny 3}
    \includegraphics[scale=0.15]{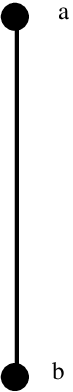}  
  } \\ 
%\hline 
  & & \hspace{0.3cm} $\longleftarrow$ &   \\ 
%\hline 
  &   &   &   \\ 
\hline 
\multirow{10}{*}{ \includegraphics[scale=0.2]{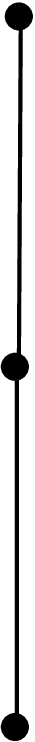}}  & 
\multirow{5}{*}{ \includegraphics[scale=0.15]{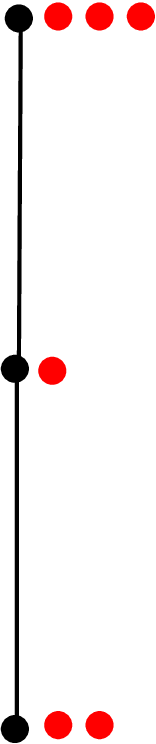}} 
 &   & \multirow{5}{*}{
 %\begin{tikzpicture}
 \psfrag{a}{\tiny 1}
 \psfrag{b}{\tiny 0}
 \psfrag{c}{\tiny 1}
  \includegraphics[scale=0.15]{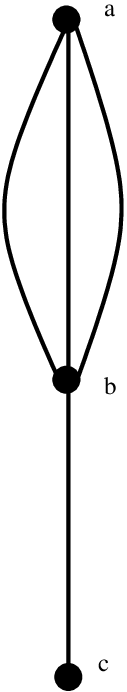}
 %\end{tikzpicture}
  }  \\ 
&&& \\
  &   & \hspace{0.3cm}  $\longleftarrow$  &  \\ 
  & &  & \\
  &   &   &   \\ 
\cline{2-4}
  & \multirow{5}{*}{ \includegraphics[scale=0.15]{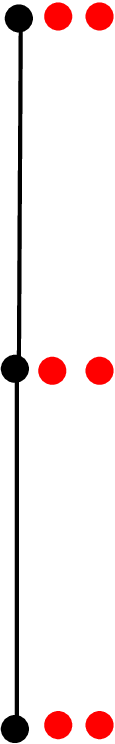}} 
  &   & \multirow{5}{*}{  
  \psfrag{a}{\tiny 1}
  \psfrag{b}{\tiny 2}
  \psfrag{c}{\tiny 1}
  \includegraphics[scale=0.15]{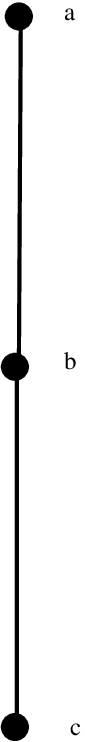}}
  \\ 
 &&& \\
 
  & & $\longleftarrow$ &   \\ 
 &&&\\
  &   &   &   \\ 
\hline 
\end{tabular} 
   \hspace{1cm}
\begin{tabular}{|c|ccc|}
\hline 
\multirow{7}{*}{
%\begin{figure}
 \includegraphics[scale=0.15]{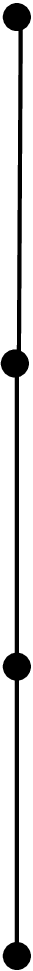}
% \end{figure}
 }  & \multirow{7}{*}{ \includegraphics[scale=0.15]{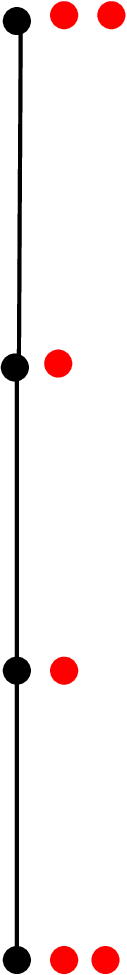}} 
 &   &  
 \psfrag{a}{\tiny 1} 
 \psfrag{b}{\tiny 0} 
 \psfrag{c}{\tiny 0} 
 \psfrag{d}{\tiny 1} 
  \multirow{7}{*}{ \includegraphics[scale=0.14]{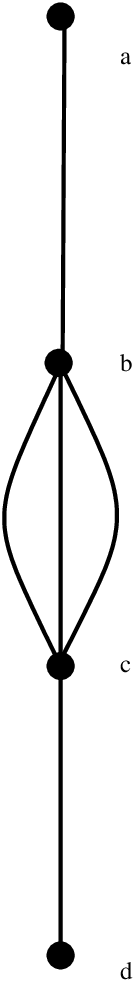}} \\ 
%\hline  
 &   &   &   \\
 &   &   &   \\
&& \hspace{0.3cm} $\longleftarrow$ \hspace{0.3cm} & \\
  &   &   &   \\ 
  &&& \\
%&&& \\
\hline 
\multirow{5}{*}{ \includegraphics[scale=0.2]{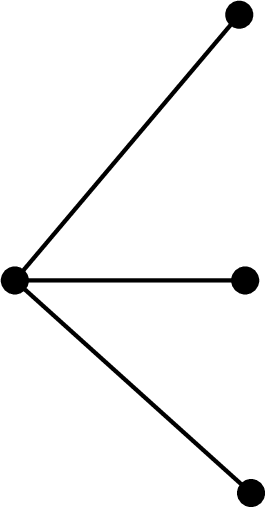}} 
 & \multirow{5}{*}{ \includegraphics[scale=0.2]{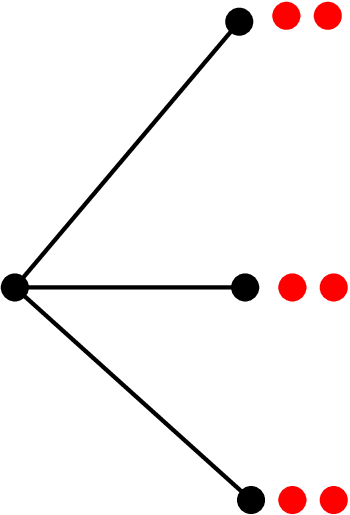}} 
 &   &  
  \psfrag{a}{\tiny 1} 
 \psfrag{b}{\tiny 1} 
 \psfrag{c}{\tiny 1} 
 \psfrag{d}{\tiny 1}  
 \multirow{5}{*}{ \includegraphics[scale=0.2]{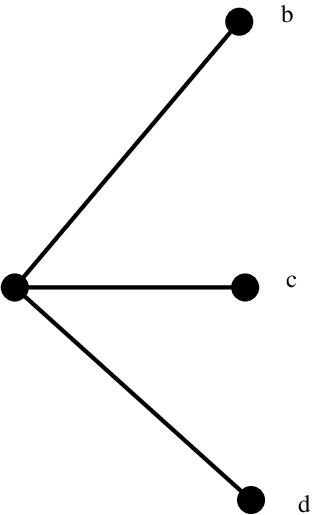}}
 \\ 
  &   &   &   \\ 
&& \hspace{0.3cm} $\longleftarrow$ \hspace{0.3cm} & \\
&&& \\
&&& \\
\hline 
\end{tabular} 
%\end{turn}
\end{center}
\caption{\small  Action $\Phi_7$ on genus $g=4$: topological strata at the boundary of the corresponding  equisymmetric locus. The group acting is   $\Z_3$ and the orbifold $\O$ is a sphere with 6 cone points of order 3. 
 The first columns show the trees $\DAbajo$, dual to all the admissible multicurves $\S$   in $\O$. 
 The left parts of the second columns show    the same trees with all possible distribution of the 6 cone points assigned to the vertices; notice that a terminal vertex   must have at least two cone points assigned, while  a degree 2 vertex   must have at least one. Since all the exponents of the generating vector are 1, no numbers are assigned to the vertices of $\DAbajo$.  
Finally,   the right parts show the graphs $\DArriba$, dual to  the preimages $\tilde \S$ of the  multicurves $\S$. The numbers are the genera of the corresponding vertex.}
\label{Tabla:Phi_7}
\end{table}

\begin{table}[ht]
 \begin{spacing}{5}
\begin{tabular}{ | c|c|m{1cm}||m{1cm}|c|c|m{0.5cm}|m{1cm}||m{1cm}|c|m{0.5cm}|m{0.5cm}|m{1cm}||m{1cm}|}
\cline{1-4}
\cline{6-9}
\cline{11-14}
  \includegraphics[scale=0.2]{arbol-1.eps}
&
  \includegraphics[scale=0.2]{arbol-1-1.eps}
&
\psfrag{a}{\tiny 111}  
\psfrag{b}{\tiny 222} 
  \includegraphics[scale=0.15]{arbol-1-a.eps}
&
\psfrag{a}{\tiny 1 }  
\psfrag{b}{\tiny 1} 
  \includegraphics[scale=0.15]{grafo-1-1-a.eps}
%\5-8
&
 
&
 %\vspace{2mm} 
  \includegraphics[scale=0.2]{arbol-2.eps}
&
 %\vspace{2mm} 
  \includegraphics[scale=0.2]{arbol-2-1.eps}
&
\psfrag{a}{\tiny 111}  
\psfrag{b}{\tiny 2 } 
\psfrag{c}{\tiny 22 } 
  \includegraphics[scale=0.15]{arbol-2-a.eps}
&
\psfrag{a}{\tiny 1  }  
\psfrag{b}{\tiny 0}
\psfrag{c}{\tiny 1 }  
  \includegraphics[scale=0.15]{grafo-2-1-a.eps}
  &
  &
 \includegraphics[scale=0.2]{arbol-2.eps}
&
  \includegraphics[scale=0.2]{arbol-2-2.eps}
&
\psfrag{a}{\tiny 11}  
\psfrag{b}{\tiny 12 } 
\psfrag{c}{\tiny 22 } 
  \includegraphics[scale=0.12]{arbol-2-a.eps}
&
\psfrag{a}{\tiny 1  }  
\psfrag{b}{\tiny 2}
\psfrag{c}{\tiny 1 }  
  \includegraphics[scale=0.12]{arbol-2-a.eps}
\\
\cline{3-4}
\cline{8-9}
\cline{13-14}
%%%%%%%%%%%%%%%%%%%%%%%%%%%%%%%%%%%%%%%% fin de la primera fila
&&
\psfrag{a}{\tiny 112 }  
\psfrag{b}{\tiny 122} 
  \includegraphics[scale=0.15]{arbol-1-a.eps}
&
\psfrag{a}{\tiny 2 }  
\psfrag{b}{\tiny 2} 
  \includegraphics[scale=0.15]{arbol-1-a.eps}
 &
 &
 &
 &
 \psfrag{a}{\tiny 112}  
\psfrag{b}{\tiny 1  } 
\psfrag{c}{\tiny 22 } 
  \includegraphics[scale=0.12]{arbol-2-a.eps}
&
\psfrag{a}{\tiny 2  }  
\psfrag{b}{\tiny 1}
\psfrag{c}{\tiny 1 }  
  \includegraphics[scale=0.12]{arbol-2-a.eps}
  &
  &
  &
  &
  \psfrag{a}{\tiny 11}  
\psfrag{b}{\tiny 22 } 
\psfrag{c}{\tiny 12 } 
  \includegraphics[scale=0.12]{arbol-2-a.eps}
&
\psfrag{a}{\tiny 1  }  
\psfrag{b}{\tiny 1}
\psfrag{c}{\tiny 0 }  
  \includegraphics[scale=0.12]{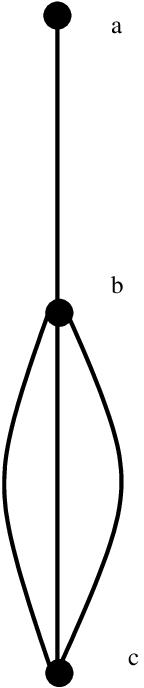}
\\
\cline{2-4}
\cline{8-9}
\cline{13-14}
%%%%%%%%%%%%%%%%%%%%%%%%%%%%%%%%%%%%%%%% fin de la segunda fila
&
 \includegraphics[scale=0.2]{arbol-1-2.eps}
&
\psfrag{a}{\tiny 11 }  
\psfrag{b}{\tiny 1222} 
  \includegraphics[scale=0.15]{arbol-1-a.eps}
&
\psfrag{a}{\tiny 1 }  
\psfrag{b}{\tiny 3} 
  \includegraphics[scale=0.15]{arbol-1-a.eps}
  &
  &&&
  \psfrag{a}{\tiny 112}  
\psfrag{b}{\tiny 2 } 
\psfrag{c}{\tiny 12 } 
  \includegraphics[scale=0.12]{arbol-2-a.eps}
&
\psfrag{a}{\tiny 2  }  
\psfrag{b}{\tiny 0}
\psfrag{c}{\tiny 0 }  
  \includegraphics[scale=0.12]{grafo-2-1-Reves-a.eps}
  &
  &&&
  \psfrag{a}{\tiny 12}  
\psfrag{b}{\tiny 11} 
\psfrag{c}{\tiny 22 } 
  \includegraphics[scale=0.12]{arbol-2-a.eps}
&
\psfrag{a}{\tiny 0  }  
\psfrag{b}{\tiny 1}
\psfrag{c}{\tiny 1 }  
  \includegraphics[scale=0.12]{grafo-2-1-a.eps}
\\
\cline{3-4}
\cline{8-9}
\cline{13-14}
 %%%%%%%%%%%%%%%%%%%%%%%%%%%%%%%%%%%%%%%% fin de la tercera fila
&&
\psfrag{a}{\tiny 12 }  
\psfrag{b}{\tiny 1222} 
  \includegraphics[scale=0.15]{arbol-1-a.eps}
&
\psfrag{a}{\tiny 0}  
\psfrag{b}{\tiny 2} 
  \includegraphics[scale=0.15]{grafo-1-1-a.eps}
  &
  &&&
  \psfrag{a}{\tiny 122}  
\psfrag{b}{\tiny 1 } 
\psfrag{c}{\tiny 12 } 
  \includegraphics[scale=0.12]{arbol-2-a.eps}
&
\psfrag{a}{\tiny 2  }  
\psfrag{b}{\tiny 0}
\psfrag{c}{\tiny 0 }  
  \includegraphics[scale=0.12]{grafo-2-1-Reves-a.eps}
  &
  &&&
  \psfrag{a}{\tiny 12}  
\psfrag{b}{\tiny 12} 
\psfrag{c}{\tiny 12 } 
  \includegraphics[scale=0.12]{arbol-2-a.eps}
&
\psfrag{a}{\tiny 0  }  
\psfrag{b}{\tiny 0}
\psfrag{c}{\tiny 0 }  
  \includegraphics[scale=0.12]{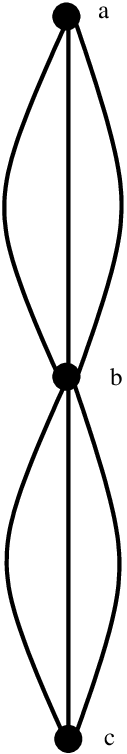}
\\
\cline{3-4}
\cline{8-9}
\cline{11-14}
%%%%%%%%%%%%%%%% Fin de la cuarta fila
&&
\psfrag{a}{\tiny 22 }  
\psfrag{b}{\tiny 1112} 
\includegraphics[scale=0.15]{arbol-1-a.eps}
&
\psfrag{a}{\tiny 1 }  
\psfrag{b}{\tiny 3} 
  \includegraphics[scale=0.15]{arbol-1-a.eps}
  &
  &&&
    \psfrag{a}{\tiny 122}  
\psfrag{b}{\tiny 2} 
\psfrag{c}{\tiny 11 } 
  \includegraphics[scale=0.12]{arbol-2-a.eps}
  &
    \psfrag{a}{\tiny 2}  
\psfrag{b}{\tiny 1} 
\psfrag{c}{\tiny 1 } 
  \includegraphics[scale=0.12]{arbol-2-a.eps}
  &
  \multicolumn{3}{c}{  }  
\\
  \cline{1-4}
\cline{8-9}
%\cline{13-14}
%%%%%%%%%%%%%% Fin de la quinta fila
 
   \multicolumn{4}{c}{  }
  %&&& 
  & &&&
     \psfrag{a}{\tiny 222}  
\psfrag{b}{\tiny 1} 
\psfrag{c}{\tiny 11 } 
  \includegraphics[scale=0.12]{arbol-2-a.eps}
  &
    \psfrag{a}{\tiny 1}  
\psfrag{b}{\tiny 0} 
\psfrag{c}{\tiny 1 } 
%\vspace{3mm}
  \includegraphics[scale=0.12]{grafo-2-1-a.eps}
 &
  \multicolumn{3}{c}{  }
 \\
 % \cline{1-4}
\cline{6-9}
%\cline{11-14}
\end{tabular}
 \end{spacing}
\vspace{-1cm}
\caption{\small Action $\Phi_8$ on genus $g=4$,  (I). The group acting is   $\Z_3$ and the orbifold $\O$ is a sphere with 6 cone points of order 3.  The two first columns (of each subtable) are as in Table \ref{Tabla:Phi_7}. 
  The third column shows the trees $\DAbajo$ with the distribution of the exponents of the generating vector. Finally, the last column shows the corresponding weighted graph $\DArriba$.
   }
\label{Tabla:Phi_8(I)}
%\end{turn}
\end{table}
\begin{table}
%\begin{center}
\begin{spacing}{4}
 \begin{tabular}{|c|p{1cm}||p{1cm}|c|c|p{2cm}||p{2cm}|}
 \cline{1-3}
 \cline{5-7}
 \multirow{4}{1cm}{
\includegraphics[scale=0.14]{arbol-3-1.eps} 
 }
 & 
 \psfrag{a}{\tiny{11}}
 \psfrag{b}{\tiny{1}}
 \psfrag{c}{\tiny{2}}
 \psfrag{d}{\tiny{22}}
 \includegraphics[scale=0.14]{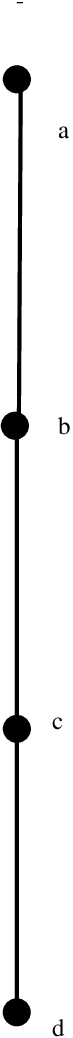}
  & 
 \psfrag{a}{\tiny{1 }}
 \psfrag{b}{\tiny{0}}
 \psfrag{c}{\tiny{0}}
 \psfrag{d}{\tiny{1}}
 \includegraphics[scale=0.14]{grafo-3-2-a.eps}  
  &   &
  \multirow{2}{1cm}{
   \psfrag{a}{\tiny{0}}
 \psfrag{b}{\tiny{0}}
 \psfrag{c}{\tiny{1}}
 \psfrag{d}{\tiny{2}}
\includegraphics[scale=0.14]{arbol-4-1.eps} 
 } & 
  
 \psfrag{b}{\tiny{11}}
 \psfrag{c}{\tiny{12}}
 \psfrag{d}{\tiny{22}}
 \includegraphics[scale=0.14]{arbol-4-a.eps}
  & 
  \psfrag{a}{\tiny{0}}
 \psfrag{b}{\tiny{1}}
 \psfrag{c}{\tiny{0}}
 \psfrag{d}{\tiny{1}}
\includegraphics[scale=0.2]{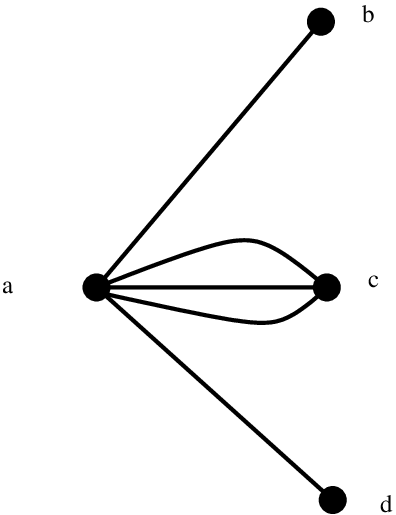} 
   \\ 
 \cline{2-3}
 \cline{6-7} 
 % Fin de la primera fila %%%%%%%%%%%%%%%%%%%%
   & 
 \psfrag{a}{\tiny{11}}
 \psfrag{b}{\tiny{2}}
 \psfrag{c}{\tiny{1 }}
 \psfrag{d}{\tiny{22}}   
   \includegraphics[scale=0.14]{arbol-3-a.eps} & 
 \psfrag{a}{\tiny{1 }}
 \psfrag{b}{\tiny{1}}
 \psfrag{c}{\tiny{1}}
 \psfrag{d}{\tiny{1}}
 \includegraphics[scale=0.14]{arbol-3-a.eps}   
    &   &  &
 
 \psfrag{b}{\tiny{12}}
 \psfrag{c}{\tiny{12}}
 \psfrag{d}{\tiny{12}}
\includegraphics[scale=0.14]{arbol-4-a.eps}    
      & 
  \psfrag{a}{\tiny{0}}
 \psfrag{b}{\tiny{0}}
 \psfrag{c}{\tiny{0}}
 \psfrag{d}{\tiny{0}}      
\includegraphics[scale=0.2]{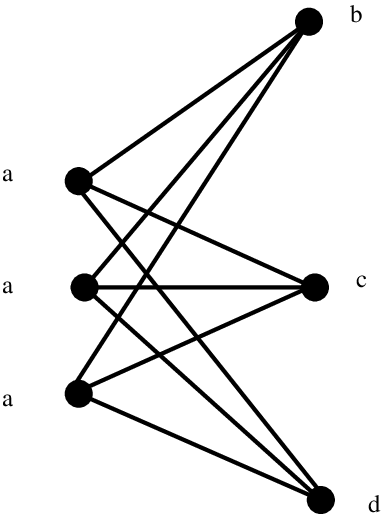}
     \\ 
 \cline{2-3}
 \cline{5-7}
 %%%%%%%%%%%%%%%%% Fin de la segunda fila %%%%%%%%%%%%%%
  & 
   \psfrag{a}{\tiny{11}}
 \psfrag{b}{\tiny{2}}
 \psfrag{c}{\tiny{2}}
 \psfrag{d}{\tiny{12}}
\includegraphics[scale=0.14]{arbol-3-a.eps}  
   &
 \psfrag{a}{\tiny{1 }}
 \psfrag{b}{\tiny{1}}
 \psfrag{c}{\tiny{0}}
 \psfrag{d}{\tiny{0}}
 \includegraphics[scale=0.14]{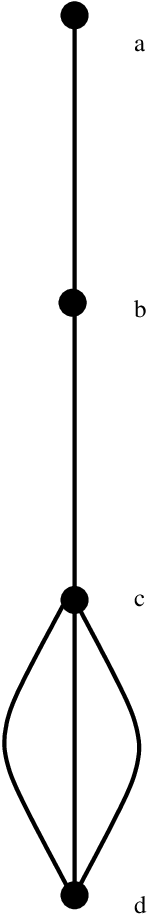}   
   & 
\multicolumn{3}{c}{ }   
    \\ 
%%%%%%%%%%%% Fin de la tercera fila %%%%%%%%%%    
    
\cline{2-3}
 
   & 
 \psfrag{a}{\tiny{12}}
 \psfrag{b}{\tiny{1}}
 \psfrag{c}{\tiny{2}}
 \psfrag{d}{\tiny{12}}
   \includegraphics[scale=0.14]{arbol-3-a.eps}
    & 
 \psfrag{a}{\tiny{0}}
 \psfrag{b}{\tiny{0}}
 \psfrag{c}{\tiny{0}}
 \psfrag{d}{\tiny{0}}
 \includegraphics[scale=0.14]{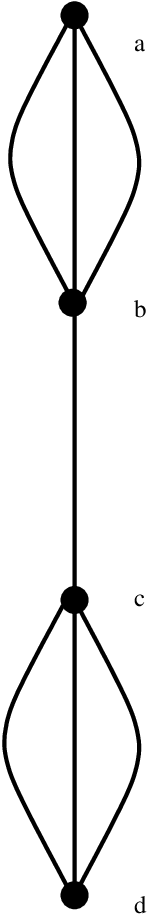}    
     & \multicolumn{3}{c}{} 
     \\ 
 \cline{1-3}
 %\cline{5-7} 
 \end{tabular} 
 \end{spacing}
 \vspace{-1cm}
\caption{Action $\Phi_8$ on genus $g=4$,  (II). Minimal strata at the boundary of the corresponding equisymmetric locus. 
 The last graph is the complete bipartite graph $K_{3,3}$.  }
\label{Tabla:Phi_8(II)}
\end{table}

% por último pongo las tablas para $\Phi_{12}$, $\Phi_{13}$ y $\Phi_{14}$.

\begin{table}
% Primera tabla, $\Phi_{12}$
\begin{tabular}{|m{1cm}||m{1cm}|}
\hline 
\multicolumn{2}{|c|}{$\Phi_{12}$}\\ 
\hline 
\psfrag{a}{\tiny 11 }  
\psfrag{b}{\tiny 12 } 
\vspace{2mm}  
  \includegraphics[scale=0.15]{arbol-1-a.eps}
& 
\psfrag{a}{\tiny 2 }  
\psfrag{b}{\tiny 2} 
\vspace{2mm}  
  \includegraphics[scale=0.15]{arbol-1-a.eps} \\ 
\hline 
 \end{tabular}
 \hspace{0.5cm}
 % 
 %Segunda tabla $\Phi_{13}$
 %
\begin{tabular}{|m{1cm}||m{1cm}|}
 \hline 
\multicolumn{2}{|c|}{$\Phi_{13}$} \\ 
 \hline 
 \psfrag{a}{\tiny 11 }  
\psfrag{b}{\tiny 44}
\vspace{2mm}  
  \includegraphics[scale=0.15]{arbol-1-a.eps}
 & 
 \psfrag{a}{\tiny  2 }  
\psfrag{b}{\tiny  2} 
\vspace{2mm} 
  \includegraphics[scale=0.15]{arbol-1-a.eps}
  \\ 
 \hline 
\psfrag{a}{\tiny 14 }  
\psfrag{b}{\tiny 14}
\vspace{2mm}  
  \includegraphics[scale=0.15]{arbol-1-a.eps}
& 
 \psfrag{a}{\tiny  0 }  
\psfrag{b}{\tiny  0} 
\vspace{2mm} 
  \includegraphics[scale=0.15]{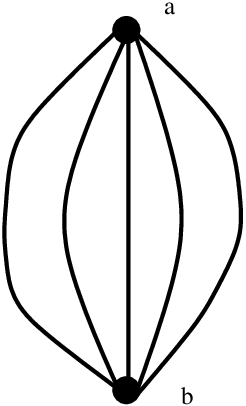}\\ 
 \hline 
\end{tabular}
\hspace{0.5cm}
%
%Tercera tabla $\Phi_{14}$
%
\begin{tabular}{|m{1cm}||m{1cm}|}
\hline 
\multicolumn{2}{|c|}{$\Phi_{14}$} \\ 
\hline 
   %%%% Fin de la primera línea %%%%%%%%%%%%%%%%%%%%%%%%
 \psfrag{a}{\tiny  12 }  
\psfrag{b}{\tiny  34} 
\vspace{3mm} 
  \includegraphics[scale=0.15]{arbol-1-a.eps}
& 
 \psfrag{a}{\tiny  2 }  
\psfrag{b}{\tiny  2} 
\vspace{3mm} 
  \includegraphics[scale=0.15]{arbol-1-a.eps}
   \\ 
   \hline 
   %%%% Fin de la segunda línea %%%%%%%%%%%%%%%%%%%%%%%%
  \psfrag{a}{\tiny  13 }  
\psfrag{b}{\tiny  24} 
\vspace{3mm} 
  \includegraphics[scale=0.15]{arbol-1-a.eps}
 &
 \psfrag{a}{\tiny  2 }  
\psfrag{b}{\tiny  2} 
\vspace{3mm} 
  \includegraphics[scale=0.15]{arbol-1-a.eps} 
  \\ 
\hline 
 \psfrag{a}{\tiny  14 }  
\psfrag{b}{\tiny  23}
\vspace{3mm}  
  \includegraphics[scale=0.15]{arbol-1-a.eps}
 & 
  \psfrag{a}{\tiny  0 }  
\psfrag{b}{\tiny  0} 
\vspace{3mm} 
 \includegraphics[scale=0.15]{grafo-1-5-a.eps}
 \\ 
\hline 
\end{tabular} 
\vspace{0.1cm}
\caption{Topological strata at the boundary of the equisymmetric loci determined by the actions $\Phi_{12}$, $\Phi_{13}$ and $\Phi_{14}$. }
\label{Tabla:Phi12Phi13Phi14}
\end{table}

In general, to find  all the topological strata at the boundary of the  equisymmetric locus $\mathcal M(\Z_p,\O,\Phi)$ determined by a cyclic $p$-gonal action, 
 one can proceed in the following way.
 \begin{itemize}
 \item[(1)] Find all the trees $\mathcal T$ with at most $r-3$ edges, where $r$ is the number of cone points of $\O$ (first column in Table \ref{Tabla:Phi_8(I)}).
 \item[(2)] Distribute $r$ points among the vertices of the tree $\mathcal T$, taking into account that for each vertex, its degree plus  the number of cone points assigned to it must be at least 3 (second column in Table \ref{Tabla:Phi_8(I)}).
 \item[(3)] Distribute the $r$ exponents of the generating vector among the $r$ points assigned to the vertices in step (2) (third column in Table \ref{Tabla:Phi_8(I)}). After these three steps  we have  determined  the   graphs $\DAbajo$ dual to all the possible admissible   multicurves $\S$ on $\O$.
 
 \item[(4)] Use Theorem \ref{thm:p-gonal} to find $\DArriba$ (fourth column in Table \ref{Tabla:Phi_8(I)}). 
 \end{itemize}

\subsubsection{Cyclotomic action.} This action is defined for any prime $p >3 $ as $\Phi\colon H_1(\O)\to \Z_p$ with generating vector $(a,a^2,a^3,\dots, a^{p-1})$. In particular, $\O_j$ has $p-1$ cone points $q_i$ and the covering surface has genus $g=\frac{(p-1)(p-3)}{2}$. 
For $p=5$, the cyclotomic action is the action $\Phi_{14}$ studied in the previous section.

Consider a multicurve $\S=\{\g_1,\dots, \g_r\}$ with $ \frac{p-1}{2}\leq r \leq p-4$ where the curve $\g_1$ surrounds the cone points $q_1,q_{p-1}$, $\g_2$ surrounds $q_2,q_{p-2}$, and so on until $\g_{\frac{p-1}{2}}$ surrounds $q_{\frac{p-1}{2}}, q_{\frac{p+1}{2}}$. Notice that $p-4$ is the maximal number of components of an admissible multicurve in $\O$.  Applying Lemma \ref{lem:p-gonal} and Theorem \ref{thm:p-gonal}, we see that  all the edges of $\DAbajo$ are multiple of $p$ and all the vertices   are  multiple of $p$ except the terminal ones (those with degree 1). The weights of all the vertices in $\DArriba$ are equal to 0. As a consequence, the graph $\DArriba$ is constructed by taking $p$ copies of $\DAbajo$ and identifying   corresponding terminal vertices. Figure \ref{fig:grafo-ciclotomico} shows an example of $\DAbajo$ and its corresponding $\DArriba$ for $p=7$.

\psfrag{DAbajo}{$\DAbajo$}
\psfrag{DArriba}{$\DArriba$}
\begin{figure}
\includegraphics[scale=0.3]{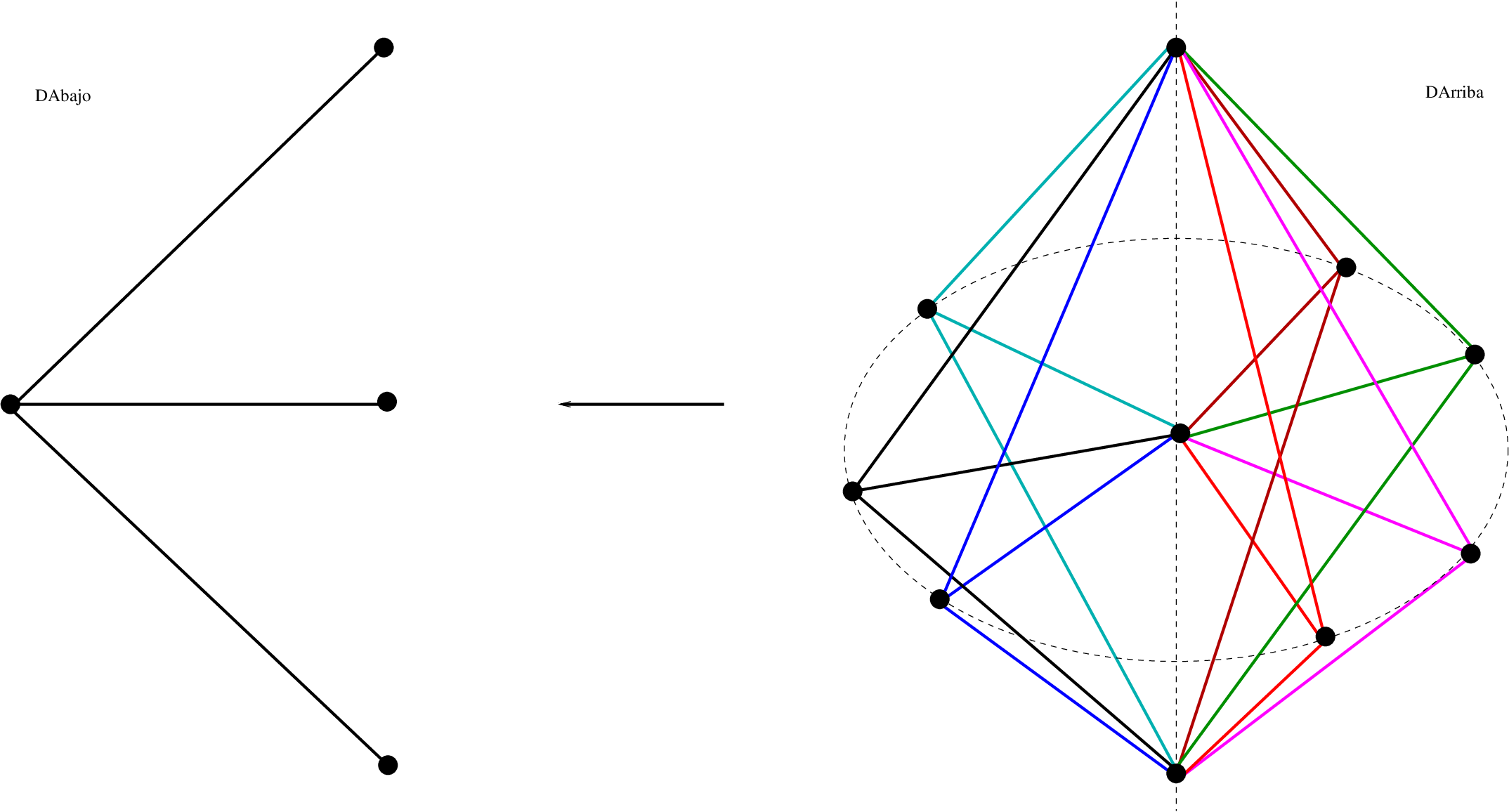}
\caption{Example of graphs $\DAbajo$ and $\DArriba$ for the cyclotomic action with $p=7$.}
\label{fig:grafo-ciclotomico}
\end{figure}

In this way, when $\S$ is a pants decomposition (i.e., when $r=p-4$), the multicurve $\tilde \S$ is an example of multicurve in $S$ equivariant under $G$ with the maximum number of connected components.

\section{Hyperelliptic locus}\label{sec:Hyperelliptic}
The hyperelliptic action is defined in the same way as the cyclic $p$-action but with $p=2$. That is, is an involution on a surface $S_{\!\!g},  $ of genus $g$ whose quotient orbifold $\O$  has genus  0, and therefore $2+2g$ cone points of order 2. It is well-known that if a Riemann  surface admits a hyperelliptic involution, this is unique, and that a Riemann surface of genus 2 always admits the hyperelliptic involution. Hence in this section $g\geq 3$. 

The hyperelliptic equisymmetric stratum is determined by the epimorphism $\Phi\colon H_1(\O)\to \Z_2$ defined as $\Phi(x_i)=a$ for all $i$, where, as in the $p$-gonal case, $x_i$ is a small loop around the cone point $q_i$.

 Let $\S\subset\O$ a multicurve. Because $\O$ is a sphere, its dual graph $\DAbajo$ is also  a tree, as in the $p$-gonal case, but now it can have semiedges, one for each arc of $\S$.  The weight of each vertex has the form $(0;2^k)$, which we will abbreviate as $k$, i.e., the weight of any vertex will be the number of cone points the corresponding orbifold  contains. The definitions of {\it even} vertex or edge is the analog to ``multiple of $p$" in the $p$-gonal case.

 We obtain a lemma totally analogous to Lemma \ref{lem:p-gonal} just by changing $p$ by $2$. The only new part appears for  the semiedges.

\begin{lem} \label{lem:hyperelliptic}
Let $\Phi\colon H_1(\O)\to \Z_2$ be the hyperelliptic   action and let $\S\subset \O$ be an admissible multicurve.
Then
\begin{itemize}
\item[(i)] An edge $\g$ of $\D_{\S}$ is even if
and only if $\g$ bounds a disc in $\O$ containing an even number of cone points.

\item[(ii)] A vertex    $\O_j$ of $\D_{\S}$ is even
 if and only if the suborbifold $\O_j$ does not contain any cone point and all the edges $\g$ incident to the vertex $\O_j$ are even.
 
 \item[(iii)] A semiedge is always odd. A semiedge incident to an odd vertex lifts to a loop. A semiedge incident to an even vertex lifts to an edge joining the two lifts of the vertex. 
\end{itemize}
\end{lem}

\begin{proof}
(i) If an edge $\g$ bounds a disc containing the cone points $q_{i_1}, \dots, q_{i_{2k}}$, then $\g $ is homotopic to $x_{i_1}\dots x_{i_{2k}}$ and thus $\Phi(\g)=a^{2k}=1$. Hence $\g$ is even. The converse is equal.  

(ii) Equal  as in Lemma \ref{lem:p-gonal}.

(iii) Let $\g$ be a semiedge of $\DAbajo$ starting at a vertex $\O_j$, which corresponds   to  an    arc $\g$  in the boundary of the  suborbifold $\O_j$ joining two cone points $q_i,q_j$. Then $\pi_1(\g)$ is generated by the loops $x_i,x_j$ and the dual curve $c_{\g}$ is  equal to one of these loops, so $\Phi(c_{\g})=a$.  Thus $\Im \Phi_{\g}=\Z_2$ and so  a semiedge is always odd. By Theorem \ref{thm:MainDiaz1}(3), the edge $E_{\g, g \mathbb Z_2 }$ joins the vertices $V_{j, g \Im \Phi_j }$ and $V_{j, g\Phi(c_{\g}) \Im \Phi_j }= V_{j, ga \Im \Phi_j }$. Thus, $\g$ lifts to a loop if $\O_j$ is odd and   $\g$ lifts to an edge joining the two lifts of $\O_j$ if $\O_j$ is even. 
\end{proof}

Finally, to describe the graph $\DArriba$ we have the analog result to Theorem \ref{thm:p-gonal}. See Figure \ref{fig:hyperelliptic} for an example of the dual graphs  of a multicurve $\S$ and of its preimage $\tilde\S$ under the hyperelliptic action. 

\psfrag{0}{\tiny 0}
\psfrag{1}{\tiny 1}
\psfrag{2}{\tiny 2}
\psfrag{3}{\tiny 3}

\begin{figure}
\center
\includegraphics[height=5cm,width=11cm]{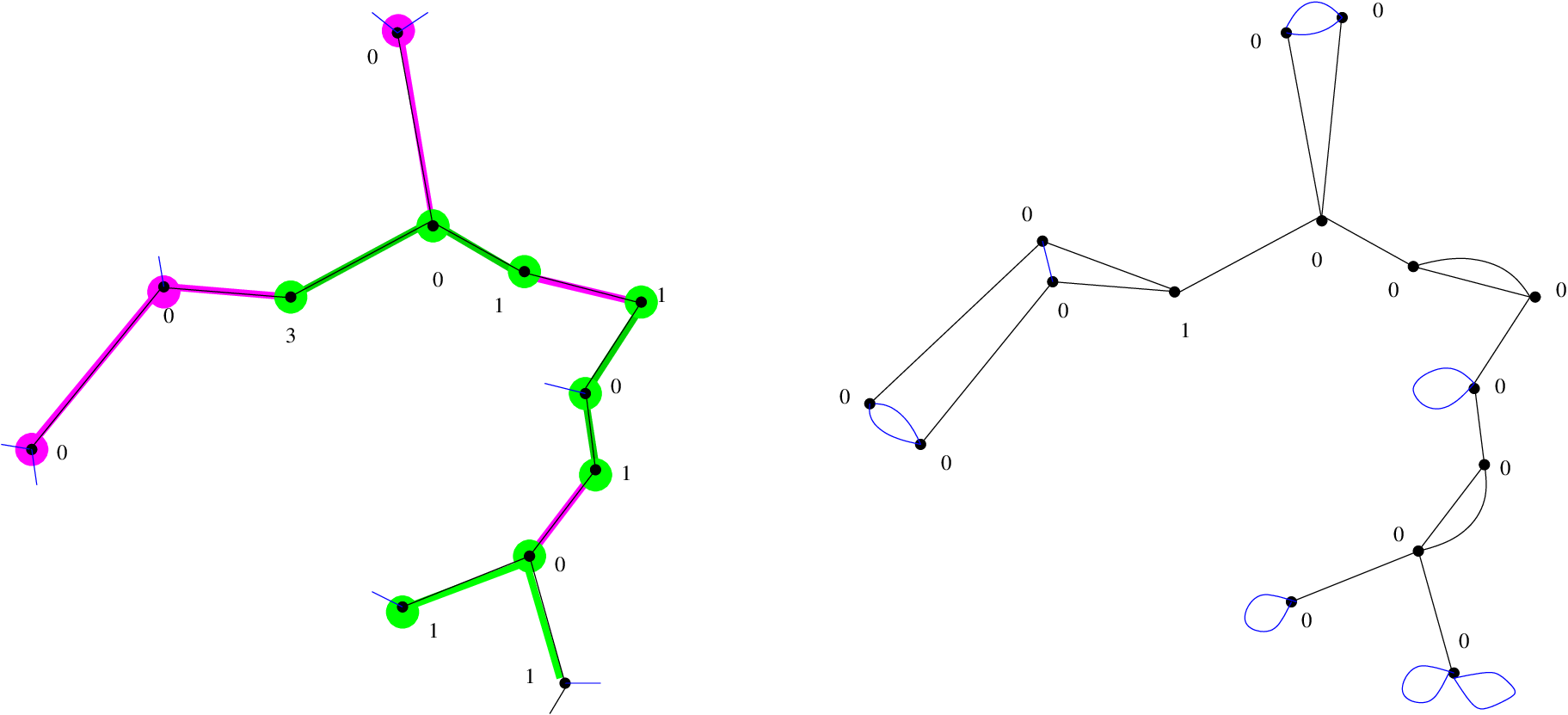}
\caption{Dual graphs  $\DAbajo, \DArriba$ of a multicurve $\S$ and of its preimage $\tilde\S$ under the hyperelliptic action on genus 12. The numbers at the vertices of $\DAbajo$ indicate the number of cone points in the corresponding suborbifold (that are not endpoints of any semiedge). Green vertices and edges are odd, hence they lift to one copy each; purple vertices and edges are even, hence lift to two copies each. Semiedges are not colored, they are always odd. 
Numbers at $\DArriba$ indicate genera.
}
\label{fig:hyperelliptic}
\end{figure}

\begin{thm} \label{thm:hyperelliptic}
Let $\Phi\colon H_1(\O ,*)\to \Z_p$ be the hyperelliptic action on genus $g$. Let $\S\subset \O$ be an admissible multicurve and let $\DAbajo$ its dual graph. Then the dual graph $\DArriba$ of $\,\tilde\S$ is obtained as $ (\D^1\sqcup \D^2)/\hspace{-0.1cm}\sim$ where $\D^1, \D^2$ are copies of $\D_{\S}$ and $\sim$ is as follows. We call $V^t_j$ to the vertex of $\D^t$ corresponding to the vertex $V_j$ of $\D_{\S}$, and similarly with the edges. Then:

\begin{enumerate}

\item  $V^1_j\sim V^2_j$ if and only if   $V_j$ is odd.

\item  An edge $E^1_j\sim E^2_j$ if and only if $E_j$ is odd.
\item A semiedge $E^1_{\g}$ is joined to $E^2_{\g}$ through their ``free   end", giving rise in $\DArriba$ to a loop or to an edge joining two different vertices as indicated in  Lemma \ref{lem:hyperelliptic}(iii). 

\item If a vertex $V=\O_j$  is even, then the degree $D_j$ of any of its preimages is equal to  the degree of $V$ and its weight $w^j$ is equal to zero.

\item  If a vertex $V=\O_j$ is odd then the  degree of any of its preimages is equal to $D_j=n_j+2(m_j+s_j)$, where  $n_j, m_j, s_j$ are, respectively the number of  odd edges, even edges and semiedges incident to $\O_j$.  Its weight is $w^j =\frac{1}{2} (n_j+c_j-2)  $, where $c_j$ is the number of cone points contained in $\O_j$.
\end{enumerate} 
\end{thm}

\begin{proof}
(1) and (2) are as in Theorem \ref{thm:p-gonal} and (3) is Lemma \ref{lem:hyperelliptic}(iii). 

(4) Let $\O_j$ be an even vertex, so $|\Im\Phi_{j}|=1$.
By Lemma \ref{lem:hyperelliptic}(ii), each edge $\g$ incident to $\O_j$ is even, so $|\Im\Phi_{\g}|=1$, and by part (iii) of that lemma,  any semiedge $\g$  is odd so $2\frac{1}{|\Im\Phi_{\g}|}=1$. Thus, by Theorem \ref{thm:MainDiaz1}(4), $D_j$ is equal to the degree of $\O_j$.

To compute the weight, by  Lemma \ref{lem:hyperelliptic}(ii), $\O_j$ does not contain any cone point, so that $\chi(\O_j)=2-deg(\O_j)$. Then using Theorem \ref{thm:MainDiaz1}(5), we obtain that $w^j=0$.

(5) Let $V_j=\O_j$ be an odd vertex, so $|\Im\Phi_{j}|=2$. By Theorem \ref{thm:MainDiaz1}, we have that
$$
D_j= 2(n_j\frac{1}{2}+m_j \frac{1}{1}+ 2s_j \frac{1}{2})=n_j+ 2(m_j+s_j).
$$
The weight of any preimage of $V_j$ is equal to
$$
w^j=1-\frac{1}{2}\left(2\left(2-n_j-m_j-s_j  -c_j\frac{1}{2}\right)+n_j+ 2(m_j+s_j)\right)=\frac{1}{2} (n_j+c_j-2).
$$
\end{proof}

\section{Action by an abelian group}\label{sec:Abelian}

In this section we consider any action by an abelian group and study the preimage of a multicurve  consisting on just one curve $\S=\{\g\}$. This will provide the topological strata of maximal dimension in the boundary of the corresponding equisymmetric loci.  We study three cases, depending on whether $\g$ is a separating or non-separating simple closed curve or is an arc. 

\subsection{Case 1:  separating simple closed curve.}

 In this case we will show in Proposition \ref{prop:Abelian-I} that the dual graph of the preimage is a complete bipartite multigraph. 

\noindent {\bf Notation.} We denote by $K_{m,n}^d$ the complete bipartite graph of order $d$  between a subset of  $m$ vertices and a subset of $n$ vertices, i.e., each vertex of the first subset is joined to each vertex of the second subset with $d$ edges.

In the proof of Proposition \ref{prop:Abelian-I} we will need the following result about abelian groups which says that if an abelian group is generated by two subgroups, then any   class  with respect to  the first subgroup intersects any class with the second subgroup in a class with respect the intersection of both subgroups.

\begin{lem}\label{lem:IntersecClassesAbelian}
Let $G$ be an abelian group. Let $H_1,H_2$ be subgroups such that $G=H_1 H_2$. Then for all $a,b\in G$ there is $c\in G$ such that  $aH_1\cap bH_2 = c(H_1\cap H_2)$. As a consequence, $|G|=\frac{|H_1||H_2|}{|H_1\cap H_2|}$.
\end{lem}
\begin{proof}

Suppose first that $b=1$. Since $G=H_1H_2$, $a=a_1a_2$ with $a_i\in H_i$. Let us check that $aH_1\cap H_2= a_2(H_1\cap H_2)$. Indeed, on the one hand, if $h\in H_1\cap H_2$, we have that $a_2h\in H_2$ and $a_2h=a a_1^{-1}h\in aH_1$. On the other hand take $y\in aH_1\cap H_2$, i.e., $y=ak_1$ for some $k_1\in H_1$ and $y=k_2$ for some $k_2\in H_2$. We want to show that $y\in a_2(H_1\cap H_2)$ or equivalently that $a_2^{-1}y\in  (H_1\cap H_2)$. But 
$ 
a_2^{-1}y= a_2^{-1}k_2\in H_2$, and 
$$
a_2^{-1}y=a_2^{-1}ak_1=a_2^{-1}a_1a_2k_1=a_1k_1\in H_1,
$$
 so we are done. 

For the general case, applying the previous argument to $a'=ab^{-1} $ we have that  $a'H_1\cap H_2= a'' (H_1\cap H_2)$ for some $a''\in G$. Thus 
$
b(a'H_1\cap H_2)= ba'' (H_1\cap H_2)  
$, i.e., $
aH_1\cap bH_2= ba'' (H_1\cap H_2)  
$.

As a consequence, any class $ aH_1$ is a union of  
$[G:H_2]$ 
classes with respect to $H_1\cap H_2$, and so 
$$
[G:(H_1\cap H_2)] = [G:H_1][G:H_2]
$$
and we easily obtain the result. 
\end{proof}

\begin{prop}\label{prop:Abelian-I}
Consider a topological action on a surface described by an epimorphism $\Phi\colon \pi_1(\O)\to G$, with $G$ an abelian group. Let $\g\subset\O$ be a closed curve separating $\O$ in two suborbifolds $\O_1,\O_2$ and consider the multicurve
  $\S=\{\g\}$. Then  $\DArriba $ is isomorphic to the graph $K_{m,n}^d$, where $m=\frac{|G|}{|\Im \Phi_1|}$, $n=\frac{|G|}{|\Im \Phi_2|}$  and $d=\frac{r }{mn}$, with  $r=\frac{|G|}{|\Im \Phi_{\g}|}$.
\end{prop}

\begin{proof} 
Since $G$ is abelian, we can work either with fundamental groups or with homology groups.  
 
By Van-Kampen theorem (or Mayer-Vietoris sequence), we have  that $H_1(\O) $ is generated by the subgroups $H_1(\O_1),H_1(\O_2)$. Also notice that   $H_1(\O_1)\cap H_1(\O_2)=H_1(\g)$. Since $\Phi$ is surjective, these properties pass to $G$, precisely, 
 $G$ is generated by $\Im\Phi_1,\Im\Phi_2$ and $\Im\Phi_1\cap \Im\Phi_2 \supset \Im \Phi_{\g}$.

By  Theorem \ref{thm:MainDiaz1}, the number of preimages of the vertex $ \O_j$ is $|\frac{G}{\Im \Phi_j}|, j=1,2$, and the number of preimages of $\g$  is $|\frac{G}{\Im\Phi_{\g}}|$.  Since the edge  $\g$ in $\DAbajo$ joins $\O_1,\O_2$, each preimage of $\g$  joins a preimage of $\O_1$ with a preimage of $\O_2$. 
 Thus $\DArriba $ is bipartite.

Given two vertices $ V_{1,a\Im\Phi_1},  V_{2,b\Im\Phi_2}$, we will show that there is some edge $E_{\g,c\Im \Phi_{\g}}$ joining them.
Indeed,  by Theorem \ref{thm:MainDiaz1}(3), an edge  $E_{\g,c\Im\Phi_{\g}}$ joins   $V_{1,a\Im\Phi_1}$ and  $V_{2,b\Im\Phi_2}$ if and only if $c\in a\Im\Phi_1$ and $c\Phi(c_{\g})\in b\Im\Phi_2$. Notice that   $c_{\g}$ is trivial (because $\D_{\S}$ is a tree), so the above is equivalent to 
 $$
 c\in a\Im\Phi_1\cap b\Im\Phi_2 =c'(\Im\Phi_1\cap \Im\Phi_2),
 $$
for some $c'\in G$, by  Lemma \ref{lem:IntersecClassesAbelian}. Since
$\Im\Phi_{\g}$ is a subgroup of $\Im\Phi_1\cap \Im\Phi_2$, the class $c'(\Im\Phi_1\cap \Im\Phi_2)$ is divided into $d'=\frac{|\Im\Phi_1\cap \Im\Phi_2|}{|\Im\Phi_{\g} |}$ classes with respect to $\Im\Phi_{\g}$. Thus there are $d'$ edges joining $ V_{1,a\Im\Phi_1},  V_{2,b\Im\Phi_2}$. Finally, by Lemma \ref{lem:IntersecClassesAbelian}, we have
$$
d' = \frac{|\Im\Phi_1\cap \Im\Phi_2|}{|\Im\Phi_{\g} |} = \frac{|\Im\Phi_1|\,|\Im\Phi_2|}{|G||\Im\Phi_{\g} |} = \frac{|\Im\Phi_1|}{|G| } \frac{|\Im\Phi_2|}{|G| } \frac{|G|}{ |\Im\Phi_{\g} | } = \frac{r}{mn}=d
$$
\end{proof}

 \noindent
 {\bf Remark.} In the situation of Proposition \ref{prop:Abelian-I} but with $G$ non-abelian,   the graph $\DArriba$ is also contained in a graph of the form $K^{d'}_{m,n}$, but is not clear to be  complete. 
 
\subsection{Case II: non-separating simple closed curve.}

 In this case we will show in Proposition \ref{prop:Abelian-II} that if $\S$ is a multicurve consisting on a non-separating simple closed curve, then  the dual graph of its  preimage is a multiple cycle $C^d_n$, that is, 
  a cycle of length $n$ and order $d$, i.e., each edge is multiple of order $d$.

\begin{prop}\label{prop:Abelian-II}
Consider a topological action on a surface described by an epimorphism $\Phi\colon \pi_1(\O)\to G$, with $G$ an abelian  group. Let $\g\subset\O$ be a non-separating simple closed curve and consider the multicurve
  $\S=\{\g\}$ and the suborbifold $\O_1=\O\setminus\g$. Then  $\DArriba $ is isomorphic to the graph $C_{m }^d$, where $m=\frac{|G|}{|\Im \Phi_1|}$, $d= \frac{|\Im \Phi_1| }{|\Im \Phi_{\g}|} $.
\end{prop}
\begin{proof}
The graph $\DAbajo$ consists on one vertex and one loop. By Theorem \ref{thm:MainDiaz1}, the graph $\DArriba$ has $m$ vertices and $ \frac{|G|}{|\Im \Phi_{\g}|}$  edges. Let $h=\Phi(c_{\g})$. Because $G$ is finite, some power of $h$ is contained in $\Im \Phi_{\g}$.  Let $r$ be the minimum positive number such that $h^r\in\Im\Phi_1$.

The fundamental group $\pi_1(\O)$ is generated by $\pi_1(\O_1)$ and by $c_{\g}$, and hence $G$ is generated by $\Im\Phi_1$ and by $\langle h\rangle $. Because $G$ is abelian, the above means that any  $g\in G$ can be written as $g=h^th_1$, for some integer $t$ and some $h_1\in \Im\Phi_1$. In other words,
$$
G=1\Im\Phi_1 \cup h\Im\Phi_1 \cup \dots \cup h^{r-1}\Im\Phi_1.
$$
Notice that any two of the above classes are different, since otherwise it would contradict the definition of $r$. In particular $r=m$.

Now, the edge $E_{\g,1\Im\Phi_{\g}}$ joins $V_{1,1\Im\Phi_1}$ with $V_{1,h\Im\Phi_1}$; the edge $E_{\g,h\Im\Phi_{\g}}$ joins $V_{1,h\Im\Phi_1}$ with $V_{1,h^2\Im\Phi_1}$; ... the edge $E_{\g,h^{r-1}\Im\Phi_{\g}}$ joins $V_{1,h^{r-1}\Im\Phi_1}$ with $V_{1,h^r\Im\Phi_1}=V_{1,1\Im\Phi_1}$. In this way we obtain a cycle of $r$ different vertices, so all the vertices are  in a cycle.

Finally, let us see that the above cycle is multiple of order $d$. First, notice that  if $E_{\g,c\Im\Phi_{\g}}$ joins $V_{1,a\Im\Phi_1}$ with $V_{1,b\Im\Phi_1}$, then $E_{\g,cz\Im\Phi_{\g}}$ joins the same vertices, for all $z\in\Im\Phi_1$ (indeed, if  $c\in a\Im\Phi_1$ and $ch\in b\Im\Phi_1$, then $cz\in a\Im\Phi_1$ and $czh\in b\Im\Phi_1$, since $G$ is abelian).
Thus, each pair of vertices joined by an edge are actually joined by $\frac{|\Im\Phi_1|}{|\Im\Phi_{\g}|} $ different edges.
Since the total number of edges of $\DArriba$ is $ \frac{|G|}{|\Im \Phi_1|}  \frac{|\Im \Phi_1| }{|\Im \Phi_{\g}|} =md$, we have the result.
\end{proof}

\subsection{Case III: multicurve is an arc.}
Finally in this section we study the case where the multicurve consists in just one arc. 

\begin{prop}\label{prop:Abelian-III}
Consider a topological action on a surface described by an epimorphism $\Phi\colon \pi_1(\O)\to G$, with $G$ an abelian group. Let $\g\subset\O$ be an arc and consider the multicurve
  $\S=\{\g\}$. Then  $\DArriba $ consists on
  \begin{itemize}
  \item[(i)] either a single vertex with $\frac{|G|}{2}$ or $\frac{|G|}{4}$  loops; or
  \item[(ii)] two different vertices with $\frac{|G|}{2}$ or $\frac{|G|}{4}$ joining them.
   \end{itemize}
\end{prop}
\begin{proof}
The curve $\g$ is an arc joining two cone points of order 2, and therefore $\Im\Phi_{\g}$ is generated by two elements of order 2. Since $G$ is abelian, the only possibilities is that $\Im\Phi_{\g}\sim \Z_2$ or $\Im\Phi_{\g}\sim \Z_2\times \Z_2$.

Notice that  $\pi_1(\O) $ is generated by $\pi_1(\O_1) $ and by a loop $\a$  surrounding one of the endpoints of $\g$. This is so because the the loop surrounding $\g$, which is the product of the 2 generators of $\pi_1(\g)$, also belongs to $\pi_1(\O_1)$. Taking images by $\Phi$ we have that $G$ is generated by $\Im\Phi_1$ and by $a=\Phi(\a) $.
As a consequence we have:
\begin{itemize}
\item[(i)] If $a\in\Im\Phi_1$, i.e., $\Im\Phi_1=G$, then $\DArriba$ has 1 vertex. The number of edges is $\frac{|G|}{2}$ or $\frac{|G|}{4}$ depending on whether $\Im\Phi_{\g}$ is isomorphic, respectively,  to $\Z_2$ or  $\Z_2\times \Z_2$.  Since there is only one vertex, all the edges are loops.
\item[(ii)]  If $a\not\in\Im\Phi_1$, then $G=\Im\Phi_1\cup a \,\Im\Phi_1$ and so   $\DArriba$ has 2 vertices. The number of edges is exactly as before.
It is left to see whether or not  each edge  joins the two different vertices.
To see this, notice that we can take  $c_{\g}=\a$, and so $\Phi(c_{\g})=a$.  The edge $E_{\g,c\Im\Phi_{\g}}$ joins $V_{1,c\Im\Phi_1}$ and $V_{1,ca\Im\Phi_1}$. The lateral classes $c\Im\Phi_1, ca\Im\Phi_1$  are always different because $c(ca)^{-1}=a^{-1} \not\in \Im\Phi_1$. Thus, any edge joins the 2 different vertices.
\end{itemize}
\end{proof}

\noindent
{\bf Remark.} In the situation of Proposition \ref{prop:Abelian-III} but with $G$ non-abelian, more options   can appear. For instance, in \cite{diaz1} it is shown that in the pyramidal action with $G=D_n$ the dihedral group of order $2n$, any graph  with a single vertex and  $k$ loops with  $k$ dividing $n$ appears as one of the graphs $\DArriba$.

  \subsection{Examples.} \label{sec:Examples-1curve}
  We show examples proving  that all the   graphs appearing in Propositions \ref{prop:Abelian-I}, \ref{prop:Abelian-II} and \ref{prop:Abelian-III} actually occur, that is, each of them is at the boundary of the equisymmetric locus of some abelian action.

\noindent
{\bf Examples 1. } Let $\O$  be an orbifold of signature $(0; 2^4,3^2,5^2)$. Denote by $q_i, i=1,\dots, 8$ its singular points with order, respectively, 2,2,5,5,2,2,3,3, and by $x_i, i=1,\dots, 8$ small loops around the $q_i$, all of them with the same orientation. Consider the epimorphism $\Phi$ defined by 
$$
\begin{array}{ccc}
\Phi\colon H_1(\O)& \to & \Z_2\times\Z_3\times\Z_5\cr
x_1 & \mapsto & (1,0,0) \cr
x_2 & \mapsto & (1,0,0) \cr
x_3 & \mapsto & (0,0,1) \cr
x_4 & \mapsto & (0,0,4) \cr
x_5 & \mapsto & (1,0,0) \cr
x_6 & \mapsto & (1,0,0) \cr
x_7 & \mapsto & (0,1,0) \cr
x_8 & \mapsto & (0,2,0) \cr
\end{array}
$$
It is clear that $\Phi$ determines a topological action since (in additive notation) $\sum_{i=1}^8\Phi(x_i)=(0,0,0)$.
 Now, consider a closed curve $\g$ separating $\O$ in two suborbifolds $\O_1,\O_2$, one of them containing the  points $x_1,x_2,x_3,x_4$ and the other the remaining ones. Let $\S=\{\g\}$. Then  we have that $\Im\Phi_1$ is isomorphic to $\Z_2\times \Z_5$; $\Im\Phi_2$ is isomorphic to $\Z_2\times \Z_3$ and $\Im\Phi_{\g}$ is trivial. Thus, by Proposition \ref{prop:Abelian-I}, $\DArriba$ is the graph $K^2_{3,5}$. We remark that the genus of the covering surface is $g=45$.

 For the general case  
consider an orbifold  $\O$  with signature $(0; d^{2d},m^{2m},n^{2n})$. Let $x_1,\dots, x_{2d}$ be loops around the singular points of order $d$,   $y_1,\dots, y_{2m}$ loops around the singular points of order $m$, and  $z_1,\dots, z_{2n}$ loops around the singular points of order $n$,   all of them with the same orientation. Let $\Phi\colon H_1(\O)  \to   \Z_d\times\Z_m\times\Z_n$ the epimorphism defined by
 $$
 \Phi(x_i)=(1,0,0),\quad  \Phi(y_i)=(0,1,0),\quad   \Phi(z_i)=(0,0,1).
 $$
 The order of $\Phi(x_i)$ is $d$, the  order of $\Phi(y_i)$ is $m$ and  order of $\Phi(z_i)$ is $n$; also the product of all $\Phi(x_i), \Phi(y_i), \Phi(z_i)$ is trivial. So $\Phi$ determines a topological  action with the covering surface of genus $g=1+dmn(d+m+n-4)$. 

 Now, consider a closed curve $\g$ separating $\O$ in two suborbifolds $\O_1,\O_2$, one of them containing the  points $x_1,\dots, x_d, y_1,\dots, y_{2m}$ and the other the remaining ones. Let $\S=\{\g\}$. Then  we have that $\Im\Phi_1$ is isomorphic to $\Z_d\times \Z_m$, $\Im\Phi_2$ is isomorphic to $\Z_d\times \Z_n$ and $\Im\Phi_{\g}$ is trivial. Thus, by Proposition \ref{prop:Abelian-I}, $\DArriba$ is the graph $K^d_{m,n}$.

\medskip

\noindent
{\bf Examples 2.}  (i) For $d>1$, let $\O$ be an orbifold of signature $(1; d,d)$ and let $\a,\b,x_1,x_2$ be a basis for its homology, where $x_1,x_2$ are loops  with the same orientation surrounding the cone points. Consider the epimorphism $\Phi\colon H_1(\O)\to \Z_d\times\Z_m$ defined as
$$
\Phi(\a)=(0,0), \quad   \Phi(\b)=(0,1), \quad  \Phi(x_1)=(1,0), \quad  \Phi(x_2)=(d-1,0).
$$
Then  $\Phi$ is a surface kernel epimorphism which determines a topological action with branching data $(1;d,d)$ and covering surface of genus $g=1+m(d-1)$.
Let $\g=\a$ and consider the multicurve $\S=\{\g\}$. Then $\Im\Phi_1$ is isomorphic to $\Z_d$, and $\Im\Phi_{\g}$ is trivial. Hence, by Proposition  \ref{prop:Abelian-II}, $\DArriba$ is the graph $C^d_{m}$.

(ii) For $d=1$, take an orbifold $\O$ of signature $(g;  \, )$ with $g\geq 2$, and let $\a_i,\b_i$ be  a homology basis. Consider the epimorphism $\Phi\colon H_1(\O)\to \Z_m\times\Z_{a }$ defined as
$$
\Phi(\a_1)=(1,0), \quad   \Phi(\a_i)=\Phi(\b_j)=(0,1), \hbox{ for all  } i>1 \hbox{  and for all  } j.
$$
Then $\Phi$ is a surface kernel epimorphism. Taking $\g=\b_1$, we have that $\Im \Phi_1=\Im\Phi_{\g}=\langle (0,1)\rangle$. Then, by Proposition \ref{prop:Abelian-II}, $\DArriba$ is the graph $C^1_{m}$.
 \medskip

 \noindent
 {\bf Examples 3.} 
  (i)  Let $\O$ be an orbifold of signature $(1;2^4)$ and let $\a,\b,x_1,x_2,x_3,x_4$ a homology basis, where the $x_i$ are the  loops surrounding the cone points, all with the same orientation. Let $\Phi\colon H_1(\O) \to    \Z_2\times\Z_a\times \Z_b$ be  the epimorphism defined as
 $$
 \Phi(\a)=( 0,1,0), \quad   \Phi(\b)=( 0,0,1), \quad  \Phi(x_1)=\Phi(x_2)=\Phi(x_3)=\Phi(x_4)=(1, 0,0).
 $$
 Let $\g$ be an arc joining the cone points $q_1,q_2$. Then $\Im\Phi_1=G$, $\Im\Phi_{\g}\approx  \Z_2$, so that $\DArriba$ has one vertex and $|G|/2$ loops. In this case the covering surface has genus $g= 1+2ab$.

 (ii) Let $\O$ be as before. Let $\Phi\colon H_1(\O) \to \Z_2\times \Z_2\times\Z_a\times \Z_b$ the epimorphism defined as
$$
 \begin{array}{ll }
 \Phi(\a)=(0,0,1,0), & \Phi(\b)=(0,0,0,1), \cr
 \Phi(x_1)=\Phi(x_2)=(1,0,0,0),  &\Phi(x_3)=\Phi(x_4)=(0,1, 0,0). 
 \end{array}
 $$
 Let $\g$ be an arc joining the cone points $q_1,q_3$. Then $\Im\Phi_1=G$, $\Im\Phi_{\g}\approx\Z_2\times\Z_2$, so that $\DArriba$ has one vertex and $|G|/4$ loops. In this case the covering surface has genus $g= 1+4ab$.

 (iii) Let $\O$ be an orbifold of signature $(1;2^2)$ and let $\a,\b,x_1,x_2$ a basis of homology. Let $\Phi\colon H_1(\O) \to \Z_2\times\Z_a\times \Z_b$ the epimorphism defined as
 $$
 \Phi(\a)=(0,1,0), \; \Phi(\b)=(0,0,1),\; \Phi(x_1)=(1,0,0),\; \Phi(x_2)=(1,0,0).
 $$
 Let $\g$ be an arc joining the two cone points and let $\S=\{\g\}$. Then  $\Im\Phi_1\approx\Z_a\times \Z_b$ and $\Im\Phi_{\g}\approx \Z_2$. Then $\DArriba$ has 2 vertices and $|G|/2$ edges joining them. In this case the covering surface has genus $g= 1+ab$.

 (iv) Let $\O$ be an orbifold with signature $(0;2^5)$ and let $x_i$ loops around the cone points $q_i$, all with the same orientation. 
Let $\Phi\colon H_1(\O) \to \Z_2\times\Z_2\times \Z_2$ the epimorphism defined as
 $$
 \Phi(x_1)=(1,1,0), \quad \Phi(x_2)=(1,0,1),\quad \Phi(x_3)=(0,1,1),\quad \Phi(x_4)=\Phi(x_5)=(0,0,1).
 $$
  Let $\g$ be an arc joining the   cone points  $q_1,q_2$ and let $\S=\{\g\}$. Then $\Im\Phi_{\g}=\langle (0,1,1),(1,0,1)\rangle$ and $\Im\Phi_1=\langle (0,1, 0),(0,0,1)\rangle$. Hence $\DArriba $ has two vertices and $2=\frac{|G|}{4}$ edges joining them. One can construct more examples from this one, for instance,  by adding genus $g$ to the orbifold $\O$, taking $G=\Z_2^3\times \Z_{a_1}\times \dots \times \Z_{a_g}$ and defining $\Phi$ for the extra generators $\a_i,\b_i$ as  $\Phi(\a_i)=\Phi(\b_i)=(0,\dots, 0,1,0, \dots, 0)$ where the $1$ is in the $(3+i)$-th place. We see that $|\Im\Phi_1|$ contains all the elements $(0,\dots, 0,1,0, \dots, 0)$ where the 1 is in the $i$-th place for any $i>1$. Thus, $\frac{|G|}{|\Im\Phi_1|} =2$ and $\DArriba$ has 2 vertices. On the other hand, 
  $$
  \Im\Phi_{\g}=\langle(1,1,0,\dots, 0), (1,0,1,0\dots, 0)\rangle,
  $$
  so it has 4 elements. Hence $\DArriba$ has  $\frac{|G|}{4} =2a_1\dots a_g$ edges joining the 2 vertices.

%%%%%%%%%%%%%%%%%%%%%%%%%%%%
  %%%%%%%%%%%%%%%%%%%%%%%%

\end{document}